\documentclass[12pt,twoside,english]{article}
\usepackage{amsfonts,babel,amssymb,amsmath,amsthm}
\usepackage{graphicx}

\newtheorem{proposition}{Proposition}[section]

\newtheorem{lemma}[proposition]{Lemma}
\newtheorem{theorem}[proposition]{Theorem}

\newtheorem{remark}{Remark}[section]

\newcommand{\punto}{\,\,\cdot\,\,}
\newcommand{\ds}{\displaystyle}
\newcommand{\smallfrac}[2]{{\textstyle\frac{#1}{#2}}}
\newcommand{\jump}[1]{[\![#1]\!]}
\newcommand{\ave}[1]{\{\!\!\{#1\}\!\!\}}
\newcommand{\room}{\phantom{\Big|}}

\title{Some properties of layer potentials and boundary integral operators for the wave equation}

\author{V\'\i ctor Dom\'\i nguez\footnote{Dep. Ingenier\'\i a Matem\'atica e Inform\'atica, E.T.S.I.I.T., Universidad P\'ublica de Navarra. 31500 - Tudela, Spain. {\tt victor.dominguez@unavarra.es}. Research partially
supported by Project MTM2010-21037} \& 
Francisco--Javier Sayas\footnote{Department of Mathematical Sciences, University of Delaware, Newark, DE 19716, USA –
{\tt fjsayas@math.udel.edu}}}

\date{\today}

\begin{document}

\maketitle

\begin{abstract}
In this work we establish some new estimates for layer potentials of the acoustic wave equation in the time domain, and for their associated retarded integral operators. These estimates are proven using time-domain estimates based on theory of evolution equations and improve known estimates that use the Laplace transform. 
\end{abstract}

\section{Introduction}

In this paper we prove some new bounds for the (two and three dimensional) time domain acoustic wave equation layer potentials and their related boundary integral operators. 

In 1986, Alain Bamberger and Tuong Ha--Duong published two articles (references \cite{BamHaD86a} and \cite{BamHaD86b}) on retarded integral equations for wave propagation. These seminal papers established much of what is known today about retarded layer potentials, proving continuity of layer potentials and their associated integral operators as well as invertibility properties of some relevant integral operators. The analysis of both papers has two key ingredients: (a) the time variable is dealt with by using a Laplace transform; (b)  estimates in the Laplace domain are proved using variational techniques in free space, very much in the spirit of \cite{NedPla73} (see also \cite{McL00}). Even if the results in \cite{BamHaD86a} and \cite{BamHaD86b} are given only for the three dimensional case (retarded operators with no memory), because of the way the analysis is given, all results can be easily generalized to any space dimension. An additional aspect that is relevant in \cite{BamHaD86a} and \cite{BamHaD86b} is the justification of time-and-space Galerkin discretization of some associated retarded boundary integral equations, a result that sparked intense activity in the French numerical analysis community on integral methods for acoustic, electromagnetic and elastic waves in the time domain. Not surprisingly, when Lubich's convolution quadrature techniques started to be applied to retarded boundary integral equations (this happened in \cite{Lub94}), the key results of Bamberger and Ha-Duong were instrumental in proving convergence estimates for a method that relies heavily on the Laplace transform of the symbol of the operator, even though it is a marching-on-in-time scheme. The relevance of having precise bounds in the Laplace domain for numerical analysis purposes has also been expanded in more recent work at the abstract level (with the recent analysis of RK--CQ schemes in \cite{BanLub11} and \cite{BanLubMel11}) and with applications to the wave equation at different stages of discretization (\cite{LalSay09}, \cite{BanGru11}, \cite{CheMonSB})

In this paper we advance in the project of developing the theory of retarded layer potentials with a view on creating a systematic approach to the analysis of CQ-BEM (Convolution Quadrature in time and Boundary Element Methods in space) for scattering problems. As opposed to most existing analytical approaches --while partially following the approach of \cite{SaySB}--, we will use purely time-domain techniques, inherently based on groups of isometries associated to unbounded operators and on how they can be used to treat initial value problems for differential equations of the second order in Hilbert spaces. We will show how to identify both surface layer potentials with solutions of wave equations with homogeneous initial conditions, homogeneous Dirichlet conditions on a distant boundary and non-homogeneous transmission conditions on the surface where the potentials are defined. This identification will hold true for a limited time-interval, and a different dynamic equation (with a new cut-off boundary placed farther away from the original surface) has to be dealt with for larger time intervals. In its turn, this will make us be very careful with dependence of constants in all bounds with respect to the (size of the) domain. Bounds for the solution of the associated evolution equations will depend on quite general results for non-homogeneous initial value problems. A delicate point will be proving that the strong solutions of these truncated (in time and space) problems coincides with the weak distributional definitions of the layer potentials. Since the type of results we will be using are not common knowledge for persons who might be interested in this work, and due to the fact that the kind of bounds we need are not standard in the theory of $C_0$-semigroups (and, as such, cannot be located in the best known references on the subject), we will give a self-contained exposition of the theory as we need it, based on the simple idea of separation of variables, the Duhamel principle, and very careful handling of orthogonal-series-valued functions.

From the point of view of what we obtain, let us emphasize that all bounds improve results that can be proved by estimates that use the Laplace transform. Improvement happens in reduced regularity requirements and in slower growth of constants as a function of time. This goes in addition to our overall aim of widening the toolbox for analysis of time-domain boundary integral equations, which we hope will be highly beneficial for analysis of novel discretization techniques for them.

Although results will be stated and proved for the acoustic wave equation (in any dimension larger than one), {\em all results hold verbatim for linear elastic waves,} as can be easily seen from how the analysis uses a very limited set of tools that are valid for both families of wave propagation problems. Extension to Maxwell equations is likely to be, however, more involved.

The paper is structured as follows. Retarded layer potentials and their associated integral operators are introduced in Section \ref{sec:2}, first formally in their strong integral forms and as solutions of transmission problems, and then rigorously through their Laplace transforms. Section \ref{sec:3} contains the statements of the two mains results of this paper, one concerning the single layer potential and the other concerning the double layer potential. Sections \ref{sec:4} and \ref{sec:5} contain the proofs of Theorems \ref{th:3.1} and \ref{th:3.2} respectively. In Section \ref{sec:6} we use the same kind of techniques to produce two more results, much in the same spirit, concerning the exterior Steklov-Poincar\'e (Dirichlet-to-Neumann and Neumann-to-Dirichlet) operators. In Section \ref{sec:7} we compare the kind of results that can be obtained with bounds in the Laplace domain with the results of Sections \ref{sec:3} and \ref{sec:6}. In Section \ref{sec:8} we state some basic results including bounds on non-homogeneous problems associated to the wave equation with different kinds of boundary conditions; these results have been used in the previous sections. Finally, Appendix \ref{sec:A} includes the already mentioned treatment of some problems related to the wave equation by means of rigorous separation of variables.

\paragraph{Notation, terminology and background.} Given a function of a real variable with values in a Banach space $X$, $\varphi:\mathbb R \to X$, we will say that it is {\em causal} when $\varphi(t)=0$ for all $t<0$. If $\varphi$ is a distribution with values in $X$, we will say that it is causal when the support of $\varphi$ is contained in $[0,\infty)$. The space of $k$-times continuously differentiable functions $I\to X$ (where $I$ is an interval) will be denoted $\mathcal C^k(I;X)$. The space of bounded linear operators between two Hilbert spaces $X$ and $Y$ is denoted $\mathcal L(X,Y)$ and endowed with the natural operator norm. Standard results on Sobolev spaces will be used thorough. For easy reference, see \cite{AdaFou03} or \cite{McL00}. Some very basic knowledge on vector-valued distributions on the real line will be used: it is essentially limited to concepts like differentiation, support, Laplace transform, identification of functions with distributions, etc. All of this can be consulted in \cite{DauLio92}.

\paragraph{On time differentiation.}
There will be two kinds of time derivatives involved in this work: for classical strong derivatives with respect to time of functions defined in $[0,\infty)$ with values on a Banach space $X$(understanding the derivative as the right derivative at $t=0$), we will use the notation $\dot u$; for derivatives of distributions on the real line with values in a Banach space $X$, we will use the notation $u'$. Partial derivatives with respect to $t$ will only make a brief appearance in a formal argument.

\begin{remark}
If $u:[0,\infty)\to X$ is a continuous function and we define
\begin{equation}\label{eq:1.a}
(E u)(t):=\left\{ \begin{array}{ll} u(t), & t\ge 0,\\ 0, & t<0,\end{array}\right.
\end{equation}
then $Eu$ defines a causal $X$-valued distribution. If $u\in \mathcal C^1([0,\infty);X)$ and $u(0)=0$, then $(Eu)'=E\dot u$. Also, if $u$ is an $X$-valued distribution and $X\subset Y$ with continuous injection, then $u$ is a $Y$-valued distribution and their distributional derivatives are the same, that is, when we consider the $X$-valued distribution $u'$ as a $Y$-valued distribution, we obtain the distributional derivative of the $Y$-valued distribution $u$. This fact is actually a particular case of the following fact: if $u$ is an $X$-valued distribution and $A\in \mathcal L(X,Y)$, then $Au$ is a $Y$-valued distribution and $(Au)'=Au'$.
\end{remark}

\section{Retarded layer potentials}\label{sec:2}

Let $\Omega_-$ be a bounded open set in $\mathbb R^d$ with Lipschitz boundary $\Gamma$ and let $\Omega_+:=\mathbb R^d\setminus\overline{\Omega_-}$. We assume that the set $\Omega_+$ is connected. No further hypothesis concerning the geometric setup will be made in this article. The normal vector field on $\Gamma$, point from $\Omega_-$ to $\Omega_+$ will be denoted $\boldsymbol\nu$.

\paragraph{Classical integral form of the layer potentials.}
For densities $\lambda, \varphi: \Gamma \times \mathbb R \to \mathbb R$ that are causal as functions of their real variable  (time), we can define the retarded single layer potential by
\[
(\mathcal S*\lambda)(\mathbf x,t):=\int_\Gamma \frac{\lambda(\mathbf
y,t-|\mathbf x-\mathbf y|)}{4\pi |\mathbf x-\mathbf y|}\mathrm
d\Gamma(\mathbf y),
\]
and the retarded double layer potential by
\begin{eqnarray*}
(\mathcal D*\varphi)(\mathbf x,t)\!\!&:=&\!\! \int_\Gamma
\nabla_{\mathbf y}\left( \frac{\varphi(\mathbf z,t-|\mathbf
x-\mathbf y|)}{4\pi |\mathbf x-\mathbf y|}\right) \Big|_{\mathbf
z=\mathbf
y}\cdot\boldsymbol\nu(\mathbf y)\mathrm d\Gamma(\mathbf y)\\
\!\!&=& \!\!\!\!\int_\Gamma \!\!\frac{(\mathbf x-\mathbf
y)\cdot\boldsymbol\nu(\mathbf y)}{4\pi |\mathbf x-\mathbf y|^3}\big(
\varphi(\mathbf y,t-|\mathbf x-\mathbf y|)+|\mathbf x-\mathbf
y|\dot\varphi(\mathbf y,t-|\mathbf x-\mathbf y|)\big) \mathrm
d\Gamma(\mathbf y).
\end{eqnarray*}
These are valid formulas for $\mathbf x\in \mathbb R^3\setminus\Gamma$ as long as the densities are smooth enough.
The two dimensional layer potentials are defined by
\begin{eqnarray*}
(\mathcal S*\lambda)(\mathbf x,t)&:=& \frac1{2\pi}\int_\Gamma
\int_0^{t-|\mathbf x-\mathbf y|}\frac{\lambda(\mathbf
y,\tau)}{\sqrt{(t-\tau)^2-|\mathbf x-\mathbf y|^2}}\,\mathrm
d\Gamma(\mathbf y) \,\mathrm d\tau
\end{eqnarray*}
and
\begin{eqnarray*}
(\mathcal D*\varphi)(\mathbf x,t)\!\!&:=&
\!\!\frac{1}{2\pi}\int_\Gamma \frac{\varphi(\mathbf
u,t-|\mathbf x-\mathbf y|)}{|\mathbf x-\mathbf y|}\,
\frac{(\mathbf x-\mathbf y)\cdot\boldsymbol\nu(\mathbf y)}{
\sqrt{(t-\tau)^2-|\mathbf x-\mathbf y|^2}}\mathrm d\Gamma(\mathbf y)\\
& & \!\!- \frac{1}{2\pi}\int_\Gamma \int_0^{t-|\mathbf
x-\mathbf y|} \hspace{-1cm}\frac{\varphi(\mathbf
y,\tau)}{(t-\tau)^2-|\mathbf x-\mathbf y|^2}\, \frac{(\mathbf
x-\mathbf y)\cdot\boldsymbol\nu(\mathbf y)}{
\sqrt{(t-\tau)^2-|\mathbf x-\mathbf y|^2}}\mathrm
d\Gamma(\mathbf y)\mathrm d\tau.
\end{eqnarray*}
Convolutional notation for potentials and operators will be used throughout. As we will shortly see, the convolution symbol makes reference to the time-convolution.

\paragraph{Layer potentials via transmission problems.} In a first step, layer potentials can be understood as solutions of transmission problems. Let $\gamma^-$ (resp. $\gamma^+$)  denote the operator that restricts functions on $\Omega_-$ (resp. $\Omega_+$) to $\Gamma$, i.e., the interior (resp. exterior) trace operator. Let similarly $\partial_\nu^\pm$ denote the interior and exterior normal derivative operators. Jumps across $\Gamma$ will be denoted
\[
\jump{\gamma u}:=\gamma^-u-\gamma^+ u, \qquad \jump{\partial_\nu
u}:=\partial_\nu^- u-\partial_\nu^+ u,
\]
while averages will be denoted 
\[
\ave{\gamma u}:=\smallfrac12(\gamma^- u+\gamma^+ u), \qquad
\ave{\partial_\nu u}:=\smallfrac12(\partial_\nu^- u+\partial_\nu^+
u).
\]
Given a causal density $\lambda$, the single layer potential $u:=\mathcal S*\lambda$ can be formally defined as the solution to the transmission problem:
\begin{subequations}\label{eq:2.b}
\begin{alignat}{4}
\label{eq:2.b.a}
u_{tt}=\Delta u & \qquad & & \mbox{in $\mathbb R^d\setminus\Gamma \times (0,\infty)$},\\
\label{eq:2.b.b}
\jump{\gamma u}=0 & & & \mbox{on $\Gamma\times (0,\infty)$},\\
\label{eq:2.b.c}
\jump{\partial_\nu u} = \lambda& & & \mbox{on $\Gamma\times (0,\infty)$},\\
\label{eq:2.b.d}
u(\punto,0) =0 & & & \mbox{in $\mathbb R^d\setminus\Gamma$},\\
\label{eq:2.b.e}
u_t(\punto,0)=0 & & & \mbox{in $\mathbb R^d\setminus\Gamma$}.
\end{alignat}
\end{subequations}
Similarly, for a causal density $\varphi$, $u:=\mathcal D*\varphi$ is the solution of the transmission problem:
\begin{alignat*}{4}
u_{tt}=\Delta u & \qquad & & \mbox{in $\mathbb R^d\setminus\Gamma \times (0,\infty)$},\\
\jump{\gamma u}=-\varphi & & & \mbox{on $\Gamma\times (0,\infty)$},\\
\jump{\partial_\nu u} = 0& & & \mbox{on $\Gamma\times (0,\infty)$},\\
u(\punto,0) =0 & & & \mbox{in $\mathbb R^d\setminus\Gamma$},\\
u_t(\punto,0)=0 & & & \mbox{in $\mathbb R^d\setminus\Gamma$}.
\end{alignat*}
With this definition, it follows that
\begin{equation}\label{eq:2.a}
\jump{\gamma (\mathcal S*\lambda)}=0, \qquad \jump{\partial_\nu
(\mathcal D*\varphi)}=0,\qquad
\jump{\partial_\nu (\mathcal S*\lambda)}=\lambda, \qquad
\jump{\gamma(\mathcal D*\varphi)}=-\varphi.
\end{equation}
The definition of the layer potentials through transmission problems 
allows us to define the following four retarded boundary integral  operators:
\begin{eqnarray}
\label{eq:2.e.a}
\mathcal V*\lambda &:=& \ave{\gamma (\mathcal
S*\lambda)}=\gamma^-(\mathcal S*\lambda)=\gamma^+(\mathcal
S*\lambda),\\
\label{eq:2.e.b}
\mathcal K^t*\lambda&:=&\ave{\partial_\nu (\mathcal S*\lambda)},\\
\label{eq:2.e.c}
\mathcal K*\varphi &:=& \ave{\gamma(\mathcal D*\varphi)},\\
\label{eq:2.e.d}
\mathcal W*\varphi &:=&-\ave{\partial_\nu (\mathcal
D*\varphi)}=-\partial_\nu^- (\mathcal
D*\varphi)=-\partial_\nu^+(\mathcal D*\varphi).
\end{eqnarray}
These definitions and the jump relations \eqref{eq:2.a} prove then that
\[
\partial_\nu^\pm (\mathcal S*\lambda) =\mp
\smallfrac12\lambda+\mathcal K^t*\lambda \qquad \gamma^\pm(\mathcal
D*\varphi)=\pm\smallfrac12\varphi+\mathcal K*\varphi.
\]

\paragraph{Layer potentials via their Laplace transforms.} Although the definition of the layer potentials through the transmission problems they are due to satisfy leads to an easy formal introduction of potentials, integral operators and most of the associated Calder\'on calculus with integral operators, properties of these operators are usually obtained by studying their Laplace transforms. This is the usual rigorous way of introducing these potentials (see \cite{BamHaD86a}, \cite{BamHaD86b}). In order to do this, consider the fundamental solution of the operator $\Delta -s^2$ for $s\in \mathbb C_+:=\{ s\in \mathbb C\,:\: \mathrm{Re}\,s>0\}$:
\[
E_d(\mathbf x,\mathbf y;s):= 
\left\{\begin{array}{ll}
\frac\imath4H^{(1)}_0(\imath s |\mathbf x-\mathbf y|), & (d=2),\\
\ds \frac{e^{-s\,|\mathbf x-\mathbf y|}}{4\pi |\mathbf x-\mathbf
y|},& (d=3).
\end{array}\right.
\]
The theory of layer potentials for elliptic problems (see \cite{Cos88} or the more general introduction in the monograph \cite{McL00}) can then be invoked in order to define the single and double layer potentials, which are weak forms of the integral expressions
\[
H^{-1/2}(\Gamma)\ni \lambda\longmapsto \mathrm S(s) \lambda := \int_\Gamma E_d(\punto,\mathbf y;s)
\lambda(\mathbf y)\,\mathrm d\Gamma(\mathbf y),
\]
and
\[
H^{1/2}(\Gamma)\ni \varphi \longmapsto \mathrm D(s)\varphi:=\int_\Gamma
\nabla_{\mathbf y} E_d(\punto,\mathbf y;s)\cdot\boldsymbol\nu(\mathbf
y)\, \varphi(\mathbf y)\mathrm d\Gamma(\mathbf y),
\]
respectively. For an arbitrary open set $\mathcal O$, we let
\[
H^1_\Delta(\mathcal O):=\{ u\in H^1(\mathcal O)\,:\,\Delta u\in L^2(\mathcal O)\},
\]
endowed with its natural norm. Then $\mathrm S(s):H^{-1/2}(\Gamma)\to H^1_\Delta(\mathbb R^d\setminus\Gamma)$ and $\mathrm D(s): H^{1/2}(\Gamma) \to  H^1_\Delta(\mathbb R^d\setminus\Gamma)$  are bounded for all $s\in \mathbb C_+$. The jump relations
\begin{equation}\label{eq:2.c}
\jump{\gamma\mathrm S(s)\lambda}=0, \quad \jump{\partial_\nu\mathrm D(s)\varphi}=0, \qquad \jump{\partial_\nu\mathrm S(s)\lambda}=\lambda,\qquad \jump{\gamma\mathrm D(s)\varphi}=-\varphi,
\end{equation}
justify the definition of the four associated boundary integral operators using averages of the traces
\begin{eqnarray*}
\mathrm V(s)\lambda:=\ave{\gamma\mathrm S(s)\lambda}=\gamma^\pm\mathrm S(s)\lambda, & &  \mathrm K^t(s)\lambda:=\ave{\partial_\nu \mathrm S(s)\lambda},\\
 \mathrm K(s)\varphi:=\ave{\gamma\mathrm D(s)\varphi}, & & \mathrm W(s)\varphi:=-\ave{\partial_\nu\mathrm D(s)\varphi}=-\partial_\nu^\pm \mathrm D(s)\varphi.
\end{eqnarray*}
Bounds of the operator norms of the two potentials and four operators above, explicited in terms of $s$, have been obtained in \cite{BamHaD86a, BamHaD86b} and \cite{LalSay09}. Using them, it is then possible to use Payley-Wiener's theorem (see an sketch of the theory in \cite{DauLio92} or a full introduction in \cite{Tre67})  and show that all six of them ($\mathrm S, \mathrm D,\mathrm V, \mathrm K,\mathrm K^t$ and $\mathrm W$) are Laplace transforms of operator-valued causal distributions. For instance, it follows that there exists an $\mathcal L(H^{-1/2}(\Gamma), H^1_\Delta(\mathbb R^d\setminus\Gamma))$-valued causal distribution $\mathcal S$ whose Laplace transform is well defined in $\mathbb C_+$ and is equal to $\mathrm S(s)$. The theory of vector-valued distributions proves then that for any causal $H^{-1/2}(\Gamma)$-valued distribution $\lambda$, the convolution product $\mathcal S*\lambda$ is a well defined causal $H^1_\Delta(\mathbb R^d\setminus\Gamma)$-valued distribution. Moreover, if $u:=\mathcal S*\lambda$, then 
\begin{equation}\label{eq:2.d}
u''=\Delta u.
\end{equation}
(Recall notation for distributional derivatives given at the end of the introductory section.) The Laplace operator in \eqref{eq:2.d} is the Laplacian $\Delta : H^1_\Delta(\mathbb R^d\setminus\Gamma)\to L^2(\mathbb R^d\setminus\Gamma)\equiv L^2(\mathbb R^d)$ and \eqref{eq:2.d} is to be understood as the equality of two $L^2(\mathbb R^d)$-valued causal distributions.
The fact that $u$ is causal and that differentiation is understood for distributions defined on the real line (as opposed to distributions defined in $(0,\infty)$), encodes the vanishing initial conditions \eqref{eq:2.b.d} and \eqref{eq:2.b.e}.
The jump properties of $\mathrm S(s)$ in \eqref{eq:2.c} prove then the transmission conditions in \eqref{eq:2.a}. This gives full justification for understanding $u=\mathcal S*\lambda$ as a solution of the transmission problem \eqref{eq:2.b} with time differentiation (and initial conditions) re-understood as differentiation of vector-valued distributions.
If $\mathcal V$ and $\mathcal K^t$ are the causal operator-valued distributions whose Laplace transforms are $\mathrm V(s)$ and $\mathrm K^t(s)$ respectively, then their time convolutions with a given causal density $\lambda$ satisfy the identities \eqref{eq:2.e.a} and \eqref{eq:2.e.b} thus identifying the two possible definitions of the time domain integral operators associated to the single layer potential.

The same considerations can be applied for a rigorous definition of the double layer potential in the sense of convolutions of vector-valued distributions. Note that both layer potentials had been introduced directly (without using the Laplace transform) in the three dimensional case in \cite{LalSay09b}, with a theory that cannot be easily extended to the two-dimensional case.

\paragraph{Propagation, uniqueness and integral representation.} Finite speed of propagation of the waves generated by layer potentials will be a key ingredient in our theoretical setting. For our purposes, only the following aspect will be used. Henceforth we take a fixed $R>0$ such that
\begin{equation}\label{eq:2.f}
\overline{\Omega^-}\subset B_0:=B(\mathbf 0;R):=\{\mathbf x\in \mathbb R^d\,:\,|\mathbf x|<R\}.
\end{equation}
We also consider the distance between $\Gamma$ and $\partial B_0$:
\begin{equation}\label{eq:2.g}
\delta:=\min\{ |\mathbf x-\mathbf y|\,:\,\mathbf x\in \Gamma, \mathbf y\in \partial B_0\}.
\end{equation}
For $T> 0$, we denote $B_T:=B(\mathbf 0;R+T)$ and we let $\gamma_T$ be the trace operator from $H^1(B_T\setminus\Gamma)$ to $H^{1/2}(\partial B_T)$. 

\begin{proposition}\label{prop:2.1}
Let $\lambda$ be an $H^{-1/2}(\Gamma)$-valued causal distribution,  $\varphi$ an $H^{1/2}(\Gamma)$-valued causal distribution, and $u:=\mathcal S*\lambda+\mathcal D*\varphi$.
\begin{itemize}
\item[{\rm (a)}] The temporal support of the $H^{1/2}(\partial B_T)$-valued distribution $\gamma_T u$ is contained in $[T+\delta,\infty)$.
\item[{\rm (b)}] Letting $O_T:=\mathbb R^d\setminus\overline{B_{T-\delta/2}}$, the temporal support of the
$H^1(O_T)$-valued distribution $u|_{O_T}$ is contained in $[T+\delta/2,\infty)$.
\end{itemize}
\end{proposition}

\begin{proof} This result is a consequence of some simple techniques related to the Laplace transform. Firstly,   if the Laplace transform $\mathrm F(s)$ of a distribution $f$ can be bounded as
\begin{equation}\label{eq:2.h}
\|\mathrm F(s)\|\le C \exp(- c\, \mathrm{Re}\,s)|s|^\mu \qquad \forall s\in \mathbb C \qquad \mbox{ with $\mathrm{Re}\,s>0$}, 
\end{equation}
where $c>0$ and $\mu \in \mathbb R$, then the support of $f$ is contained in $[c,\infty)$. Using estimates of the fundamental solution $E_d$ as a function of $s$, it is possible to prove a bound like \eqref{eq:2.h} for $\mathrm S(s)$ (resp. $\mathrm D(s)$) as an operator from $H^{-1/2}(\Gamma)$ (resp. $H^{1/2}(\Gamma)$) to  $H^{1/2}(\partial B_T)$ and to $H^1(O_T)$.
\end{proof}

\begin{proposition}\label{prop:2.2}
Let $\lambda$ be an $H^{-1/2}(\Gamma)$-valued causal distribution and $\varphi$ an $H^{1/2}(\Gamma)$-valued causal distribution $\varphi$ and assume that both are Laplace transformable. Then $u:=\mathcal S*\lambda-\mathcal D*\varphi$ is the only causal $H^1(\mathbb R^d\setminus\Gamma)$-valued distributional solution of the transmission problem
\[
u''=\Delta u, \qquad \jump{\gamma u}=\varphi, \qquad \jump{\partial_\nu u}=\lambda
\]
that admits a Laplace transform.
\end{proposition}

\section{Main results}\label{sec:3}

Before stating the two main results of this paper, we need to make precise statements on some constants related to the geometric setting and the Sobolev norms. The reference radius $R>0$ will be chosen so that \eqref{eq:2.f} holds.

Given an open set $\mathcal O$, we will denote
\[
\| u\|_{\mathcal O}:=\Big(\int_{\mathcal O}|u(\mathbf x)|^2\mathrm d\mathbf x\Big)^{1/2}, \qquad \| u\|_{1,\mathcal O}^2:=\Big(\| u\|_{\mathcal O}^2+\| \nabla u\|_{\mathcal O}^2\Big)^{1/2}.
\]
The first set of constants we need are the Poincar\'e-Friedrichs constants on the balls $B_T:=B(\mathbf 0; R+T)$ for $T\ge 0$:
\begin{equation}\label{eq:3.a}
\| v\|_{B_T}\le C_T \| \nabla v\|_{B_T} \qquad \forall v \in H^1_0(B_T).
\end{equation}
A simple scaling argument shows that $C_T=C_0(1+T/R)$. 
The second relevant constant is a continuity constant for the interior and exterior trace operators. It will be jointly expressed for functions that are $H^1$ on each side of $\Gamma$:
\begin{equation}\label{eq:3.b}
\| \gamma^\pm u\|_{1/2,\Gamma}\le C_\Gamma \| u\|_{1,B_0\setminus\Gamma} \qquad \forall u \in H^1(B_0\setminus\Gamma).
\end{equation}
Here $\|\punto\|_{1/2,\Gamma}$ is a fixed determination of the $H^{1/2}(\Gamma)$-norm (several equivalent choices are available in the literature; see \cite{AdaFou03} or \cite{McL00}). The third constant is related to a lifting of the trace operator. Since $\gamma^- :H^1(\Omega_-)\to H^{1/2}(\Gamma)$ is bounded and surjective, there exists a bounded linear operator $L^-:H^{1/2}(\Gamma) \to H^1(\Omega_-)$ such that
$\gamma^-L^-\varphi=\varphi$ for all $\varphi\in H^{1/2}(\Gamma),$
i.e., $L^-$ is a bounded right-inverse of the interior trace. We then denote $C_L:=\| L^-\|$. The use we will make of this operator and its norm will be through $L:H^{1/2}(\Gamma)\to H^1(\mathbb R^d\setminus\Gamma)$ given by
\[
L \varphi :=\left\{ \begin{array}{ll} L^- \varphi & \mbox{in $\Omega_-$},\\
0 & \mbox{in $\Omega_+$},\end{array}\right.
\]
noting that 
\begin{equation}\label{eq:3.c}
 \| L\varphi\|_{1,\mathbb R^d\setminus\Gamma}=\|L\varphi\|_{1,\Omega_-}\le C_L \|\varphi\|_{1/2,\Gamma}, \quad \gamma^- L\varphi=\varphi, \quad \gamma^+L\varphi=0.
\end{equation}
The final constant is related to the definition of the normal derivative. Given $u\in H^1_\Delta(B_0\setminus\Gamma)$ we can define $\partial_\nu^\pm u\in H^{-1/2}(\Gamma)$ with Green's formula. Then, there is a constant $C_\nu$ such that
\begin{equation}\label{eq:3.d}
\| \partial_\nu^\pm u \|_{-1/2,\Gamma}\le C_\nu \Big( \|\nabla u\|_{B_0\cap \Omega_\pm}^2+\|\Delta u\|_{B_0\cap \Omega_\pm}^2  \Big)^{1/2} \qquad \forall u \in H^1_\Delta(B_0\setminus\Gamma).
\end{equation}

The main theorems of this paper are given next. For simplicity of exposition, we assume that data are smooth (i.e., $\mathcal C^\infty$) and causal. Some considerations on the smoothness of data will be made in Remark \ref{remark:7.1}.

\begin{theorem}\label{th:3.1}
Let $\lambda$ be a causal smooth $H^{-1/2}(\Gamma)$-valued function and let
\[
B_2^{-1/2}(\lambda,t):= \int_0^ t \Big( \|\lambda(\tau)\|_{-1/2,\Gamma}+\| \ddot\lambda(\tau)\|_{-1/2,\Gamma}\Big)\mathrm d\tau.
\]
Then for all $t \ge 0$
\begin{eqnarray}
\label{eq:3.e}
\| (\mathcal S*\lambda)(t)\|_{1,\mathbb R^d} &\le& C_\Gamma \Big(\|\lambda(t)\|_{-1/2,\Gamma}+ \sqrt{1+C_t^2}\, B_2^{-1/2}(\lambda,t)\Big),\\
\label{eq:3.f}
\|(\mathcal V*\lambda)(t)\|_{1/2,\Gamma} &\le &  C_\Gamma^2  \Big(\|\lambda(t)\|_{-1/2,\Gamma}+ \sqrt{1+C_t^2}\, B_2^{-1/2}(\lambda,t)\Big),\\
\label{eq:3.g}
\|(\mathcal K^t*\lambda)(t)\|_{-1/2,\Gamma} &\le& \sqrt2\,C_\nu C_\Gamma
 \Big(\|\lambda(t)\|_{-1/2,\Gamma}+  B_2^{-1/2}(\lambda,t)\Big).
\end{eqnarray}
\end{theorem}

\begin{theorem}\label{th:3.2}
Let $\varphi$ be a causal smooth $H^{1/2}(\Gamma)$-valued function and let
\begin{eqnarray*}
B_2^{1/2}(\varphi,t)&:=& \int_0^ t \Big( \|\varphi(\tau)\|_{1/2,\Gamma}+\| \ddot\varphi(\tau)\|_{1/2,\Gamma}\Big)\mathrm d\tau,\\
B_4^{1/2}(\varphi,t)&:=& \int_0^ t \Big(4 \|\varphi(\tau)\|_{1/2,\Gamma}+5\| \ddot\varphi(\tau)\|_{1/2,\Gamma}+\|\varphi^{(4)}(\tau)\|_{1/2,\Gamma}\Big)\mathrm d\tau.
\end{eqnarray*}
Then for all $t \ge 0$
\begin{eqnarray}\label{eq:3.m}
\| (\mathcal D*\varphi)(t)\|_{1,\mathbb R^d\setminus\Gamma} &\le& C_L \Big(\|\varphi(t)\|_{1/2,\Gamma}+ \sqrt{1+C_t^2}\, B_2^{1/2}(\varphi,t)\Big),\\
\label{eq:3.n}
\|(\mathcal K*\varphi)(t)\|_{1/2,\Gamma} &\le &  C_\Gamma C_L  \Big(\|\varphi(t)\|_{1/2,\Gamma}+ \sqrt{1+C_t^2}\, B_2^{1/2}(\varphi,t)\Big),\\
\label{eq:3.o}
\|(\mathcal W*\varphi)(t)\|_{-1/2,\Gamma} &\le& \sqrt2\,C_\nu C_L
 \Big(4\|\varphi(t)\|_{1/2,\Gamma}+2\|\ddot\varphi(t)\|_{1/2,\Gamma}+  B_4^{1/2}(\varphi,t)\Big).
\end{eqnarray}
\end{theorem}

\section{The single layer potential}\label{sec:4}

Since the convolution operator $\lambda\mapsto \mathcal S*\lambda$ preserves causality, in order to obtain bounds at a given value of the time variable $t=T$, only the value of $\lambda$ in $(T,\infty)$ is not relevant. Therefore, we can assume without loss of generality that the growth of $\lambda$ allows it to have a Laplace transform. We can actually assume that $\lambda$ is compactly supported for the sake of the arguments that follow.

\paragraph{Introduction of a cut-off boundary.}
Let $u:=\mathcal S*\lambda$. By Proposition \ref{prop:2.2}, $u$ is a causal distribution with values in $X:=H^1(\mathbb R^d)\cap H^1_\Delta(\mathbb R^d\setminus\Gamma)$. Moreover, it is the only ($X$-valued causal distributional Laplace transformable) solution of
\begin{equation}\label{eq:4.a}
u''=\Delta u \qquad \mbox{and}\qquad \jump{\partial_\nu u}=\lambda,
\end{equation}
with the differential equation taking place in the sense of distributions with values in $L^2(\mathbb R^d\setminus\Gamma)\equiv L^2(\mathbb R^d)$, while the transmission condition is to be understood in the sense of $H^{-1/2}(\Gamma)$-valued distributions. Let now $T>0$ be fixed and let $B_T$ and $\delta$ be as in Section \ref{sec:2}.
We look for a causal distribution with values in
\[
X_T:=H^1_0(B_T)\cap H^1_\Delta(B_T\setminus \Gamma)
\]
such that
\begin{equation}\label{eq:4.b}
 u_T'' = \Delta u_T \qquad \mbox{and}\qquad \jump{\partial_\nu u_T}=\lambda.
\end{equation}
This differential equation is understandable in the sense of $L^2(B_T)$-valued distributions. We will show that for smooth data $\lambda$, this problem has strong solutions, with the time derivatives understood in the classical sense.

\begin{proposition}\label{prop:4.1}
As $H^1_\Delta(B_T\setminus\Gamma)$-valued distributions, $u=u_T$ in  $(-\infty,T+\delta)$.
\end{proposition}

\begin{proof} Consider the $H^1_\Delta(B_T\setminus\Gamma)$-valued distribution $w:=u-u_T=u|_{B_T}-u_T.$ Then
\[
w''=\Delta w, \quad \jump{\partial_\nu w}=0, \quad \mbox{and}\quad \gamma_T w=\gamma_T u. 
\]
Since the support of $\gamma_T u$ is contained in $[T+\delta,\infty)$ (by Proposition \ref{prop:2.1}), so is the support of $w$, which proves the result.
\end{proof}

\begin{proposition}\label{prop:4.2} For causal $\lambda\in\mathcal C^2(\mathbb R;H^{-1/2}(\Gamma))$, the unique solution of \eqref{eq:4.b} satisfies
\begin{equation}\label{eq:4.i}
u_T\in \mathcal C^2([0,\infty);L^2(B_T))\cap \mathcal C^1([0,\infty);H^1_0(B_T))\cap \mathcal C([0,\infty);X_T),
\end{equation}
the strong initial conditions $u_T(0)=\dot u_T(0)=0$ and the bounds for all $t\ge 0$
\begin{eqnarray}
\label{eq:4.c}
\| u_T(t)\|_{1,B_T} & \le & C_\Gamma \Big(\|\lambda(t)\|_{-1/2,\Gamma} +\sqrt{1+C_T^2}\,\, B_2^{-1/2}(\lambda,t)\Big),\\
\label{eq:4.d}
\| \nabla u_T(t)\|_{B_T} & \le & C_\Gamma \Big(\|\lambda(t)\|_{-1/2,\Gamma} + B_2^{-1/2}(\lambda,t)\Big),\\
\label{eq:4.e}
\| \Delta u_T(t)\|_{B_T\setminus\Gamma} & \le & C_\Gamma \Big(\|\lambda(t)\|_{-1/2,\Gamma} + B_2^{-1/2}(\lambda,t)\Big).
\end{eqnarray}
\end{proposition}

\begin{proof} Consider first the function $u_0:[0,\infty)\to H^1_0(B_T)$ defined by solving the steady-state problems
\[
-\Delta u_0(t)+u_0(t)=0\quad \mbox {in $B_T\setminus\Gamma$},\qquad \jump{\partial_\nu u_0(t)}=\lambda(t), \qquad \gamma_T u_0(t)=0,
\]
for $t\ge 0$. The variational formulation of this family of boundary value problems is
\[
\left[
\begin{array}{l} u_0(t) \in H^1_0(B_T),\\[1.5ex]
(\nabla u_0(t),\nabla v)_{B_T}+(u_0(t),v)_{B_T}=\langle \lambda(t),\gamma v\rangle_\Gamma\qquad \forall v \in H^1_0(B_T),
\end{array}
\right.
\]
where $\langle\cdot,\cdot\rangle_\Gamma$ is the $H^{-1/2}(\Gamma)\times H^{1/2}(\Gamma)$ duality product.
Therefore, a simple argument yields
\begin{equation}\label{eq:4.ee}
\|u_0(t)\|_{1,B_T}\le C_\Gamma\|\lambda(t)\|_{-1/2,\Gamma},\qquad \|\Delta u_0(t)\|_{B_T\setminus\Gamma}\le C_\Gamma \| \lambda (t)\|_{-1/2,\Gamma}.
\end{equation}
Note that $u_0(t)$ is the result of applying a bounded linear (time-independent) map $H^{-1/2}(\Gamma)\to X_T$ to $\lambda(t)$. Therefore, since $\lambda$ is twice continuously differentiable in $[0,\infty)$, it follows that
\begin{equation}\label{eq:4.f}
\|\ddot u_0(t)\|_{1,B_T}\le C_\Gamma\|\ddot \lambda(t)\|_{-1/2,\Gamma},\qquad \|\Delta \ddot u_0(t)\|_{B_T\setminus\Gamma}\le C_\Gamma \| \ddot \lambda (t)\|_{-1/2,\Gamma}.
\end{equation}
We next consider the function $v_0:[0,\infty) \to H^2(B_T)\cap H^1_0(B_T)$ that solves the evolution problem
\begin{equation}\label{eq:4.o}
\ddot v_0(t)= \Delta v_0(t)+u_0(t)-\ddot u_0(t)\quad t\ge 0, \qquad
v_0(0)=\dot v_0(0)=0,
\end{equation}
i.e.,  the hypotheses of Proposition \ref{prop:8.1} hold with $f= u_0-\ddot u_0$. Therefore, using \eqref{eq:4.ee}-\eqref{eq:4.f}, it follows that
\begin{equation}\label{eq:4.g}
\| \Delta v_0(t)\|_{B_T} \le \int_0^ t \| \nabla u_0(\tau)-\nabla  \ddot u_0(\tau)\|_{B_T}\mathrm d\tau \le C_\Gamma B_2^{-1/2}(\lambda,t),
\end{equation}
as well as
\begin{equation}\label{eq:4.h}
\| v_0(t)\|_{B_T} \le C_TC_\Gamma B_2^{-1/2}(\lambda,t), \qquad \|\nabla v_0(t)\|_{B_T}\le C_\Gamma B_2^{-1/2}(\lambda,t). 	
\end{equation}
If we now define $u_T:=u_0+v_0$, then the regularity requirement \eqref{eq:4.i} is satisfied and  the three bounds in the statement of the proposition are direct consequences of \eqref{eq:4.ee}, \eqref{eq:4.g} and \eqref{eq:4.h}.
Moreover,
\[
\ddot u_T(t)=\Delta u_T(t), \qquad \jump{\partial_\nu u_T(t)}=\lambda(t), \qquad \gamma_T u_T(t)=0 \qquad \forall t \ge 0.
\]
Note also that $u_T(0)=u_0(0)=0$ and $\dot u_T(0)=\dot u_0(0)=0$, since $\lambda(0)=0$ and $\dot \lambda(0)=0$ ($\lambda:\mathbb R\to H^{-1/2}(\Gamma)$ is assumed to be $\mathcal C^2$ and causal). Therefore, considering the extension operator \eqref{eq:1.a}
it follows that $E u_T$ is an $X_T$-valued causal distribution, $(Eu_T)''=E\ddot u_T=E\Delta u_T=\Delta E u_T$ and $\jump{\partial_\nu E u_T}=E\jump{\partial_\nu u_T}=E \lambda|_{(0,\infty)}=\lambda$. Therefore, $E u_T$ satisfies \eqref{eq:4.b} and the proof is finished.
\end{proof}

\paragraph{Proof of Theorem \ref{th:3.1}}By Proposition \ref{prop:2.1}, the distribution $u|_{\mathbb R^d\setminus\overline{ B_{T-\delta/2}}}$ vanishes in the time interval $(-\infty,T+\delta/2)$. Therefore, by Proposition \ref{prop:4.1}, $u_T(t)=0$ in the annular domain $B_T\setminus\overline{ B_{T-\delta/2}}$ for all $t\le T+\delta/2$. This makes the extension by zero of $u_T(t)$ to $\mathbb R^d\setminus\overline{B_T}$ an element of $H^1_\Delta (\mathbb R^d\setminus\Gamma)$ for all $t\le T+\delta/2$. (Note that the overlapping annular region is needed to ensure that the Laplace operator does not generate a singular distribution on $\partial B_T$.) Then, the argument of Proposition \ref{prop:4.1} can be used to show that the distribution $u$ can be identified with this extension in the time interval $(-\infty,T+\delta/2)$. Therefore, identifying $u(T)=u_T(T)$, the inequalities of  Proposition \ref{prop:4.2} yield
\begin{eqnarray}
\label{eq:4.j}
\|(\mathcal S*\lambda) (T)\|_{1,\mathbb R^d} & \le & C_\Gamma \Big(\|\lambda(T)\|_{-1/2,\Gamma} +\sqrt{1+C_T^2}\, B_2^{-1/2}(\lambda,T)\Big),\\
\label{eq:4.k}
\|\nabla (\mathcal S*\lambda)(T)\|_{\mathbb R^d} & \le & C_\Gamma \Big(\|\lambda(T)\|_{-1/2,\Gamma} + B_2^{-1/2}(\lambda,T)\Big),\\
\label{eq:4.l}
\| \Delta (\mathcal S*\lambda)(T)\|_{\mathbb R^d\setminus\Gamma} & \le & C_\Gamma \Big(\|\lambda(T)\|_{-1/2,\Gamma} + B_2^{-1/2}(\lambda,T)\Big).
\end{eqnarray}
We can now substitute all occurrences of $T$ by $t$, since $T$ was arbitrary. The result is now almost straightforward. First of all, \eqref{eq:4.j} is just \eqref{eq:3.e}. Also, by the trace inequality \eqref{eq:3.b} and the fact that $\mathcal V*\lambda=\gamma^\pm (\mathcal S*\lambda)$, \eqref{eq:3.f} is a direct consequence of \eqref{eq:3.e}. Finally, the bound for the normal derivative \eqref{eq:3.d}, the fact that $\mathcal K^t*\lambda=\ave{\partial_\nu (\mathcal S*\lambda)}$, and inequalities \eqref{eq:4.k}-\eqref{eq:4.l} prove \eqref{eq:3.g}.

\section{The double layer potential}\label{sec:5}

We start by introducing a cut-off boundary $\partial B_T$ as in Section \ref{sec:4} (for arbitrary $T>0$).
We are going to compare $u:=\mathcal D*\varphi$ with the causal distribution $u_T$ with values in
\[
Y_T:=\{ v \in H^1_\Delta (B_T\setminus\Gamma)\,:\, \gamma_T u=0\},
\]
such that
\begin{equation}\label{eq:5.a}
u_T''=\Delta u_T, \quad \jump{\gamma u_T}=-\varphi \quad \mbox{and}\qquad \jump{\partial_\nu u_T}=0.
\end{equation}
The same argument as the one of Proposition \ref{prop:4.1} shows that, as $H^1_\Delta(B_T\setminus\Gamma)$-valued distributions $u=u_T$ in $(-\infty,T+\delta)$, where $\delta$ is defined in \eqref{eq:2.g}. Smoothness of the solution of \eqref{eq:5.a} and bounds for it in different norms will be proved in two steps. Note that, from the point of view of regularity Proposition \ref{prop:5.2} improves the initial estimate of Proposition \ref{prop:5.1}, but that more regularity of $\varphi$ is used in the process.

\begin{proposition}\label{prop:5.1} For causal $\varphi \in \mathcal C^2(\mathbb R; H^{1/2}(\Gamma))$,
the unique solution of \eqref{eq:5.a} satisfies
\begin{equation}\label{eq:5.b}
u_T \in \mathcal C^1([0,\infty);L^2(B_T))\cap \mathcal C([0,\infty);H^1(B_T\setminus\Gamma)),
\end{equation}
the strong initial conditions $u_T(0)=\dot u_T(0)=0$ and the bounds for all $t\ge 0$
\begin{eqnarray}\label{eq:5.c}
\| u_T(t)\|_{1,B_T\setminus\Gamma} & \le & C_L \Big( \|\varphi(t)\|_{1/2,\Gamma}+ \sqrt{1+ C_T^2}\,\, B_2^{1/2}(\varphi,t)\Big),\\
\label{eq:5.d}
\| \nabla u_T(t)\|_{B_T\setminus\Gamma} & \le & C_L \Big( \|\varphi(t)\|_{1/2,\Gamma}+  B_2^{1/2}(\varphi,t)\Big).
\end{eqnarray}
\end{proposition}

\begin{proof} Let first $u_0:[0,\infty)\to H^1(B_T\setminus\Gamma)$ be given by solving the steady-state problems
\begin{subequations}\label{eq:5.h}
\begin{alignat}{4}
-\Delta u_0(t)+u_0(t)=0 \quad \mbox{in $B_T\setminus\Gamma$},  & \qquad & & \jump{\gamma u_0(t)}=-\varphi(t), \\
\gamma_T u_0(t)=0, & & &  \jump{\partial_\nu u_0(t)}=0,
\end{alignat}
\end{subequations}
for each $t\ge 0$.
The variational formulation of \eqref{eq:5.h} is
\begin{equation}\label{eq:5.k}
\left[ \begin{array}{l}
\ds u_0(t) \in H^1(B_T\setminus\Gamma),\\[1.5ex]
\jump{\gamma u_0(t)}=-\varphi(t),\quad \gamma_T u_0(t)=0,\\[1.5ex]
(\nabla u_0(t),\nabla v)_{B_T\setminus\Gamma}+(u_0(t),v)_{B_T}=0 \qquad \forall v\in H^1_0(B_T).
\end{array}\right.
\end{equation}
Using the lifting operator \eqref{eq:3.c}, we can choose the test $v=u_0(t)+L\varphi(t)\in H^1_0(B_T)$ in \eqref{eq:5.k}, and prove the estimate
\begin{equation}\label{eq:5.e}
\| u_0(t)\|_{1,B_T\setminus\Gamma} \le \| L\varphi(t)\|_{1,B_T\setminus\Gamma}\le C_L \|\varphi(t)\|_{1/2,\Gamma}.
\end{equation}
Since $u_0(t)$ is the result of applying a linear bounded (time-independent) map $H^{1/2}(\Gamma) \to Y_T$ to $\varphi(t)$, it follows that
\begin{equation}\label{eq:5.f}
\|\ddot u_0(t)\|_{1,B_T\setminus\Gamma}\le C_L \|\ddot\varphi(t)\|_{1/2,\Gamma}.
\end{equation}
We then consider $v_0:[0,\infty)\to H^1_0(B_T)$ to be a solution of
\begin{equation}\label{eq:5.i}
\ddot v_0(t)=\Delta v_0(t)+u_0(t)-\ddot u_0(t)\quad t\ge 0, \qquad v_0(0)=\dot v_0(0)=0,
\end{equation}
with the equation taking place in $H^{-1}(B_T)$ (that is, $v_0$ is a weak solution in the terminology of Section \ref{sec:8}).
By Proposition \ref{prop:8.2} (the right-hand side $f:=u_0-\ddot u_0:[0,\infty)\to L^2(B_T)$ is continuous) we can bound
\begin{equation}\label{eq:5.g}
\|\nabla v_0(t)\|_{B_T}\le \int_0^ t \| u_0(\tau)-\ddot u_0(\tau)\|_{B_T} \mathrm d\tau \le C_L B_2^{1/2}(\varphi,t),
\end{equation}
where we have applied \eqref{eq:5.e}-\eqref{eq:5.f}. 

Let us then define $u_T:=u_0+v_0$. Since $\varphi\in \mathcal C^2([0,\infty);H^{1/2}(\Gamma))$, it follows that $u_0\in \mathcal C^2([0,\infty);H^1(B_T\setminus\Gamma))$ and $v_0\in \mathcal C^1([0,\infty);L^2(B_T))\cap \mathcal C([0,\infty);H^1_0(B_T))$ by Proposition \ref{prop:8.2}. Therefore, $u_T$ satisfies \eqref{eq:5.b}. Since  $\varphi(0)=\dot\varphi(0)=0$, it follows that $u_T(0)=\dot u_T(0)=0$. Considering \eqref{eq:5.h} and \eqref{eq:5.i} (recall that $v_0$ takes values in $H^1_0(B_T)$), it follows that
\[
\ddot u_T(t)=\Delta u_T(t), \quad \jump{\gamma u_T(t)}=-\varphi(t),\quad \jump{\partial_\nu u_T(t)}=0, \quad \gamma_T u_T=0, \qquad \forall t\ge 0.
\]
Noting that $\| v_0(t)\|_{B_T}\le C_T \| \nabla v_0(t)\|_{B_T}$, and using 
\eqref{eq:5.e}, \eqref{eq:5.f}, and \eqref{eq:5.g}, it follows that  $u_T$ satisfies the bounds \eqref{eq:5.c} and \eqref{eq:5.d}. 

The delicate point of this proof lies in showing that $u_T$ can be identified with the $Y_T$-valued distributional solution of \eqref{eq:5.a}, since $v_0$ {\em is not} a continuous $Y_T$-valued function. However, $w_0(t):=\int_0^t v_0(\tau)\mathrm d\tau$ is a continuous function with values in $H^2(B_T)\cap H^1_0(B_T)$  (see Proposition \ref{prop:8.2}) and therefore in $Y_T$. We can then define $\widehat u_T:=E u_0+(E w_0)'$, which is a causal $Y_T$-valued distribution for which we can easily prove that
\[
\jump{\gamma \widehat u_T}=E\jump{\gamma u_0}+(E\jump{\gamma w_0})'=-E \varphi|_{(0,\infty)}=-\varphi
\]
and similarly $\jump{\partial_\nu \widehat u_T}=0$. Since $w_0\in \mathcal C^2([0,\infty);L^2(B_T))\cap \mathcal C([0,\infty);Y_T)$, and $w_0(0)=\dot w_0(0)=0$, it follows that $(Ew_0)''=E \ddot w_0=E \Delta w_0=\Delta E w_0$ and therefore $\widehat u_T''=(E u_0)''+(E \Delta w_0)'=\Delta E u_0+\Delta (E w_0)'=\Delta \widehat u_T$. Thus, $\widehat u_T$ satisfies \eqref{eq:5.a}. Finally, since $w_0\in \mathcal C^1([0,\infty);H^1_0(B_T))$ and $w_0(0)=0$, it is clear that, as an $H^1_0(B_T)$-valued distribution $(E w_0)'=E \dot w_0=E v_0$ and thus, as an $H^1(B_T\setminus\Gamma)$-valued distribution $\widehat u_T=E u_T$, and the bounds \eqref{eq:5.c} and \eqref{eq:5.d} are satisfied by the solution of  \eqref{eq:5.a}.
\end{proof}

\begin{proposition}\label{prop:5.2} For causal $\varphi\in \mathcal C^4(\mathbb R;H^{1/2}(\Gamma))$,
the unique solution of \eqref{eq:5.a} satisfies
\begin{equation}\label{eq:5.l}
u_T\in \mathcal C^2([0,\infty);L^2(B_T))\cap \mathcal C^1([0,\infty);H^1(B_T\setminus\Gamma))\cap \mathcal C([0,\infty);Y_T))
\end{equation}
and the bounds for all $t\ge 0$
\begin{equation}\label{eq:5.m}
\|\Delta u_T(t)\|_{B_T\setminus\Gamma} \le C_L \Big( 4\|\varphi(t)\|_{1/2,\Gamma}+2\|\ddot\varphi(t)\|_{1/2,\Gamma}+B_4^{1/2}(\varphi,t)\Big).
\end{equation}
\end{proposition}

\begin{proof}
Consider now the solution of the problems
\begin{subequations}\label{eq:5.n}
\begin{alignat}{4}
-\Delta u_1(t)+u_1(t)=L(\ddot\varphi(t)-\varphi(t)) \quad \mbox{in $B_T\setminus\Gamma$},  & \qquad & & \jump{\gamma u_1(t)}=-\varphi(t), \\
\gamma_T u_1(t)=0, & & &  \jump{\partial_\nu u_1(t)}=0,
\end{alignat}
\end{subequations}
for each $t\ge 0$, where $L$ is the lifting operator of \eqref{eq:3.c}. Using the variational formulation of \eqref{eq:5.n} and the fact that $u_1(t)+L\varphi(t)\in H^1_0(B_T)$, it follows that
\[
\| u_1(t)+L\varphi(t)\|_{1,B_T\setminus\Gamma}\le \| L\varphi(t)\|_{1,B_T\setminus\Gamma}+\| L(\ddot\varphi(t)-\varphi(t))\|_{B_T}\le C_L (2\|\varphi(t)\|_{1/2,\Gamma}+\|\ddot\varphi(t)\|_{1/2,\Gamma})
\]
and therefore
\begin{equation}\label{eq:5.o}
\| u_1(t)\|_{1,B_T\setminus\Gamma}\le C_L (3\|\varphi(t)\|_{1/2,\Gamma}+\|\ddot\varphi(t)\|_{1/2,\Gamma}).
\end{equation}
Using \eqref{eq:5.n} and \eqref{eq:5.o} it also follows that
\begin{equation}\label{eq:5.p}
\| \Delta u_1(t)\|_{B_T} \le \| u_1(t)\|_{B_T}+ \| L (\ddot\varphi(t)-\varphi(t))\|_{B_T}\le  C_L  (4\|\varphi(t)\|_{1/2,\Gamma}+2\|\ddot\varphi(t)\|_{1/2,\Gamma}).
\end{equation}
Differentiating \eqref{eq:5.o} twice with respect to $t$, it follows that
\begin{equation}\label{eq:5.q}
\| \ddot u_1(t)\|_{1,B_T\setminus\Gamma}\le C_L (3\|\ddot\varphi(t)\|_{1/2,\Gamma}+\|\varphi^{(4)}(t)\|_{1/2,\Gamma}).
\end{equation}

Consider next the evolution equation that looks for $v_1:[0,\infty)\to Y_T$ such that
\begin{equation}\label{eq:5.r}
\ddot v_1=\Delta v_1(t)+ f(t) \quad \forall t\ge 0, \qquad v_1(0)=\dot v_1(0)=0,
\end{equation}
where
\[
f(t):=u_1(t)-\ddot u_1(t)+L(\varphi(t)-\ddot\varphi(t))=\Delta u_1(t)-\ddot u_1(t).
\]
Note that $\jump{\gamma f(t)}=0$ for all $t$, and that $f:[0,\infty)\to H^1_0(B_T)$ is continuous. Moreover, by \eqref{eq:5.o} and \eqref{eq:5.q}, we can bound
\begin{equation}\label{eq:5.s}
\| \nabla f(t)\|_{B_T}\le \| f(t)\|_{1,B_T}\le C_L \Big( 4 \|\varphi(t)\|_{1/2,\Gamma}+ 5 \|\ddot\varphi(t)\|_{1/2,\Gamma} + 
\|\varphi^{(4)}(t)\|_{1/2,\Gamma}\Big).
\end{equation}
By Proposition \ref{prop:8.1}, problem \eqref{eq:5.r} has a unique (strong) solution and we can bound
\begin{equation}\label{eq:5.t}
\| \Delta v_1(t)\|_{B_T}\le \int_0^t\|\nabla f(\tau)\|_{B_T}\mathrm d\tau \le C_L  B_4(\varphi,t).
\end{equation}

If we finally define $u_T:=u_1+v_1$, the smoothness of $u_1:[0,\infty)\to Y_T$ (directly inherited from that of $\varphi$) and the regularity of $v_1$ that is derived from Proposition \ref{prop:8.1} prove that \eqref{eq:5.l} holds. The bound \eqref{eq:5.m} is a direct consequence of \eqref{eq:5.p} and \eqref{eq:5.r}. The fact that the extension $E u_T$ is the $Y_T$-valued causal distributional solution of \eqref{eq:5.a} can be proved with the same kind of arguments that were used at the end of the proof of Propositions \ref{prop:4.2}.
\end{proof} 

\paragraph{Proof of Theorem \ref{th:3.2}.} With exactly the same arguments that allowed to prove \eqref{eq:4.j}, \eqref{eq:4.k} and \eqref{eq:4.l} as a consequence of Proposition \ref{prop:4.2}, we can prove that for all $t\ge 0$
\begin{eqnarray}
\label{eq:5.u}
\|(\mathcal D*\varphi) (t)\|_{1,\mathbb R^d\setminus\Gamma} & \le & C_L \Big(\|\varphi(t)\|_{1/2,\Gamma} +\sqrt{1+C_t^2}\, B_2^{1/2}(\varphi,t)\Big),\\
\label{eq:5.v}
\|\nabla (\mathcal D*\varphi)(t)\|_{\mathbb R^d\setminus\Gamma} & \le & C_L \Big(\|\varphi(t)\|_{1/2,\Gamma} + B_2^{1/2}(\varphi,t)\Big),\\
\label{eq:5.w}
\| \Delta (\mathcal D*\varphi)(t)\|_{\mathbb R^d\setminus\Gamma} & \le & C_L \Big(4\|\varphi(t)\|_{1/2,\Gamma} 
+2\|\ddot\varphi(t)\|_{1/2,\Gamma}+ B_4^{1/2}(\varphi,t)\Big),
\end{eqnarray}
as a consequence of Propositions \ref{prop:5.1} and \ref{prop:5.2}.
The bounds of Theorem \ref{th:3.2} are now straightforward. Inequality \eqref{eq:3.m} is just \eqref{eq:5.u}, while  the fact that $\mathcal K*\varphi=\ave{\gamma (\mathcal D*\varphi)}$ and the trace  inequality \eqref{eq:3.b} prove \eqref{eq:3.m}. Finally
the bound for the normal derivative \eqref{eq:3.d}, the definition of $\mathcal W*\varphi=-\partial_\nu^\pm (\mathcal D*\varphi)$ and inequalities \eqref{eq:5.v}-\eqref{eq:5.w} prove \eqref{eq:3.o}.

\section{Exterior Steklov-Poincar\'e operators}\label{sec:6}

In this section we include bounds on the exterior Dirichlet-to-Neumann and Neumann-to-Dirichlet operators that can be obtained with the same techniques than in the previous sections. We give some details for the easier case (the Neumann-to-Dirichlet operator,whose treatment runs in parallel to that of the single layer retarded potential) in order to emphasize the need of dealing with some slightly different evolution problems as part of the analysis process.

Let us consider the bounded open set $B_T^+:=B_T \cap \Omega_+$ and the space
\begin{equation}\label{eq:6.1}
V_T:= \{ u \in H^1(B_T^+)\,:\, \gamma_T u=0\}.
\end{equation}
We can then consider the associated Poincar\'e-Friedrichs inequality
\begin{equation}\label{eq:6.2}
\| u\|_{B_T^+}\le E_T \| \nabla u\|_{B_T^+} \qquad \forall u \in V_T.
\end{equation}
Recalling that  $B_T$ is a ball with radius $R+T$, it is possible to take  $E_T \le 2(R+T)$ (see \cite[Chapter II, Section 1]{Bra07}).
Since the exterior trace operator $\gamma^+:V_0\to H^{1/2}(\Gamma)$ is surjective, it has a bounded right-inverse. By extending this right-inverse by zero to $\Omega_+\setminus \overline{B_0^+}$ we can construct $L^+ : H^{1/2}(\Gamma) \to H^1(\Omega_+)$ satisfying
\begin{equation}\label{eq:6.3}
\| L^+\varphi\|_{1,\Omega_+}=\| L^+ \varphi\|_{1,B_T^+}\le C_L^+ \|\varphi\|_{1/2,\Gamma},\qquad
\gamma^+L^+\varphi = \varphi \qquad\forall \varphi\in H^{1/2}(\Gamma).
\end{equation}
Note that, in particular, $\gamma_T L^+\varphi=0$ for all $T\ge 0$ and all $\varphi$.

\begin{theorem}\label{th:6.1} For causal $\lambda\in \mathcal C^2(\mathbb R;H^{-1/2}(\Gamma))$,  the unique causal $H^1_\Delta(\Omega_+)$-valued Laplace transformable distribution such that
\begin{equation}\label{eq:6.4}
u''=\Delta u, \qquad \partial_\nu^+ u=\lambda,
\end{equation}
satisfies the bounds
\begin{equation}\label{eq:6.5}
\| u(t)\|_{1,\Omega_+} \le C_\Gamma \Big(\|\lambda(t)\|_{-1/2,\Gamma}+ \sqrt{1+E_T^2}\, B_2^{-1/2}(\lambda,t)\Big) \qquad \forall t \ge 0.
\end{equation}
Finally the associated Neumann-to-Dirichlet operator $\mathrm{NtD}(\lambda):=\gamma^+ u$ (where $u$ is the solution of \eqref{eq:6.4}) satisfies the bounds
\begin{equation}\label{eq:6.6}
\| \mathrm{NtD}(\lambda)(t)\|_{1/2,\Gamma} \le C_\Gamma^2  \Big(\|\lambda(t)\|_{-1/2,\Gamma}+ \sqrt{1+E_T^2}\, B_2^{-1/2}(\lambda,t)\Big) \qquad \forall t \ge 0.
\end{equation}
\end{theorem}

\begin{proof}
The proof is very similar to the one of Theorem \ref{th:3.1}. By solving steady state problems, we first construct  $u_0: [0,\infty) \to H^1(B_T^+)$ satisfying
\begin{equation}\label{eq:6.7}
-\Delta u_0(t)+u_0(t)=0 \quad \mbox{in $B_T^+$}, \qquad \partial_\nu^+ u_0(t)=\lambda(t), \qquad \gamma_T u_0(t)=0.
\end{equation}
A simple argument allows us to bound
\begin{equation}\label{eq:6.8}
\| u_0(t)\|_{1,B_T^+}\le C_\Gamma \|\lambda(t)\|_{-1/2,\Gamma} \qquad \mbox{and} \qquad 
\|\ddot  u_0(t)\|_{1,B_T^+}\le C_\Gamma \|\ddot\lambda(t)\|_{-1/2,\Gamma}.
\end{equation}
This function feeds the evolution equation looking for $v_0:[0,\infty)\to D_T$ (the set $D_T$ is defined in \eqref{eq:8.14}) that satisfies
\begin{equation}\label{eq:6.9}
\ddot v_0(t)=\Delta v_0(t)+u_0(t)-\ddot u_0(t) \quad t\ge 0, \qquad v_0(t)=\dot v_0(t)=0.
\end{equation}
We now apply the general result on the wave equation with mixed boundary conditions (Proposition \ref{prop:8.3}) that guarantees the existence of a strong solution of \eqref{eq:6.9} satisfying the bounds
\begin{equation}\label{eq:6.10}
\| v_0(t)\|_{B_T^+}\le E_T C_\Gamma B_2^{-1/2}(\lambda,t), \qquad \| \nabla v_0(t)\|_{B_T^+}\le  C_\Gamma B_2^{-1/2}(\lambda,t).
\end{equation}
Adding the solutions of \eqref{eq:6.7} and \eqref{eq:6.9} we obtain a function $u_T:=u_0+v_0: [0,\infty) \to H^1_\Delta(B_T^+)\cap V_T$ satisfying $\ddot u(t)=\Delta u(t)$, $\partial_\nu u_T(t)=\lambda(t)$ and vanishing initial conditions at $t=0$. The extension $E u_T$ is then an $(H^1_\Delta(B_T^+)\cap V_T)$-valued causal distributional solution of \eqref{eq:6.4} (with the Laplace operator acting only in the bounded domain $B_T^+$). The arguments of Proposition \ref{prop:4.1} and at the beginning of the proof of Theorem \ref{th:3.1} can be applied verbatim in order to identify the function that extends $u_T(t)$ by zero to the exterior of $B_T$ with the $H^1_\Delta(\Omega_+)$-valued distributional solution of \eqref{eq:6.4} for $t\le T+\delta/2$. Finally, the bound \eqref{eq:6.5} for $t=T$ is a straightforward consequence of  \eqref{eq:6.8}, \eqref{eq:6.10}, and the identification of $u(T)=u_T(T)$, while \eqref{eq:6.6} follows from \eqref{eq:6.5} and the trace inequality \eqref{eq:3.b}.
\end{proof}

\begin{theorem}\label{th:6.2} For causal $\varphi \in \mathcal C^4(\mathbb R;H^{1/2}(\Gamma))$, the unique causal $H^1_\Delta(\Omega_+)$-valued Laplace transformable distribution such that
\begin{equation}\label{eq:6.11}
u''=\Delta u, \qquad \gamma^+ u=\varphi,
\end{equation}
satisfies the bounds
\begin{equation}\label{eq:6.12}
\| u(t)\|_{1,\Omega_+} \le C_L^+ \Big(\|\varphi(t)\|_{1/2,\Gamma}+ \sqrt{1+E_T^2}\, B_2^{1/2}(\varphi,t)\Big) \qquad \forall t \ge 0.
\end{equation}
Finally the associated Dirichlet-to-Neumann operator $\mathrm{DtN}(\varphi):=\partial_\nu^+ u$ (where $u$ is the solution of \eqref{eq:6.11}) satisfies the bounds
\begin{equation}\label{eq:6.13}
\| \mathrm{DtN}(\varphi)(t)\|_{-1/2,\Gamma} \le 
 \sqrt2\,C_\nu C_L^+
 \Big(4\|\varphi(t)\|_{1/2,\Gamma}+2\|\ddot\varphi(t)\|_{1/2,\Gamma}+  B_4^{1/2}(\varphi,t)\Big).
\end{equation}
\end{theorem}

\begin{proof} This result can be proved like Theorem \ref{th:3.2} by resorting to a double decomposition of a localized version of problem \eqref{eq:6.11} (obtained by adding a boundary condition $\gamma_T u=0$) as a sum of an adequate steady-state lifting of the Dirichlet data plus the solution of an evolution problem. The proof is almost identical to that of Theorem \ref{th:3.2}: the main difference is in the evolution problem, that now contains Dirichlet boundary conditions on $\Gamma$ as well as on $\partial B_T$ (see Remark \ref{remark:8.1}).
\end{proof}

\section{Comparison with  Laplace domain bounds}\label{sec:7}

The original analysis for the layer operators and associated integral equations, given in \cite{BamHaD86a} and \cite{BamHaD86b}, was entirely developed in the resolvent set (that is, by taking the Laplace transform). Those results can be used to derive uniform bounds similar to those of Theorems \ref{th:3.1} and \ref{th:3.2}. We next show how to obtain these estimates and show that our technique produces stronger estimates, in terms of requiring less regularity of the densities and having constants that increase less fast with respect to $t$. 

\begin{remark}\label{remark:7.1}
Before moving on, let us emphasize that Theorems \ref{th:3.1} and \ref{th:3.2} do not require the highest order derivative involved to be bounded and local integrability is enough to keep the quantities $B_2^{\pm 1/2}$ and $B_4^{1/2}$ bounded. This is as much as saying that Theorem \ref{th:3.1} is still valid if $\lambda:\mathbb R\to H^{-1/2}(\Gamma)$ is causal and $\mathcal C^1$ (therefore $\lambda(0)=\dot\lambda(0)=0$) and $\lambda''$ is locally integrable. Similarly, estimates \eqref{eq:3.m} and \eqref{eq:3.n} of Theorem \ref{th:3.2} hold for $\mathcal C^1$ causal densities with locally integrable second derivative, while the estimate \eqref{eq:3.o} holds for $\mathcal C^3$ densities with locally integrable fourth derivative.
\end{remark}

Estimates in the Laplace domain can be obtained using the following all purpose theorem, which is just a refinement of Lemma 2.2  in \cite{Lub94}. The refinement stems from taking more restrictive hypotheses; these ones are chosen in order to fit closer to what can be proved for all operators associated to the wave equation.

\begin{theorem}\label{th:7.1}
Let $f$ be an $\mathcal L(X,Y)$-valued causal distribution whose Laplace transform $\mathrm F(s)$ exists for all $s\in \mathbb C_+:=\{ s\in \mathbb C\,:\, \mathrm{Re} s>0\}$ and satisfies
\[
\|\mathrm F(s)\|_{\mathcal L(X,Y)}\le C_{\mathrm F}(\mathrm{Re}s)|s|^\mu \qquad \forall s\in \mathbb C_+,
\]
where $\mu \ge 0$ and $C_{\mathrm F}:(0,\infty)\to (0,\infty)$ is a non-increasing function. Let
\[
k:=\lfloor\mu+2\rfloor, \qquad \varepsilon:=k-(\mu+1)\in (0,1].
\]
Then for all causal $\mathcal C^{k-1}$ function $g:\mathbb R \to X$ with locally integrable $k$-th distributional derivative, the $Y$-valued distribution $f*g$ is a causal continuous function such that
\[
\| (f*g)(t)\|_Y \le \frac{\sqrt{2^{1+\varepsilon}}}{\pi\varepsilon} t^\varepsilon C_{\mathrm F}(1/t) \int_0^ t \| g^{(k)}(\tau)\|_X \mathrm d\tau \qquad \forall t\ge 0.
\]
\end{theorem}

In Table \ref{table:1} we compare regularity and growth of the bounds between what Theorems \ref{th:3.1} and \ref{th:3.2} prove and what can be obtained by a systematic analysis in the Laplace domain. The bounds in the Laplace domain are explicit or implicitly given in \cite{BamHaD86a} and \cite{BamHaD86b}. They are also collected in \cite[Appendix 2]{LalSay09}. We also include the comparison of what Theorems \ref{th:6.1} and \ref{th:6.2} assert about Steklov-Poincar\'e operators with similar results obtained through the Laplace domain analysis. 

\begin{table}[h]
\begin{center}
\begin{tabular}{ll|ll|ll}
$\mathrm F$\room & $X \to Y$ & $C_{\mathrm F}(\sigma)|s|^\mu$ & $E(t) D_k(\eta,t)$ & $n$ & $\mathcal O(t)$\\ \hline 
$\mathrm S$\room & $H^{-1/2}(\Gamma) \to H^1(\mathbb R^d)$ & $\frac{|s|}{\sigma\underline\sigma^2}$ & $t^2\max\{1,t^2\} D_3(\lambda,t)$ & 2 & $\mathcal O(t)$\\
$\mathrm V$\room & $H^{-1/2}(\Gamma)\to H^{1/2}(\Gamma)$ &  $\frac{|s|}{\sigma\underline\sigma^2}$ & $t^2\max\{1,t^2\} D_3(\lambda,t)$ & 2 & $\mathcal O(t)$\\
$\mathrm K^t$\room & $H^{-1/2}(\Gamma)\to H^{-1/2}(\Gamma)$ &  $\frac{|s|^{3/2}}{\sigma\underline\sigma^{3/2}}$ & $t^{3/2}\max\{1,t^{3/2}\} D_3(\lambda,t)$ & 2 & $\mathcal O(1)$\\
$\mathrm D$\room & $H^{1/2}(\Gamma)\to H^1(\mathbb R^d\setminus\Gamma)$ & $\frac{|s|^{3/2}}{\sigma\underline\sigma^{3/2}}$ & $t^{3/2}\max\{1,t^{3/2}\} D_3(\varphi,t)$ & 2 & $\mathcal O(t)$\\
$\mathrm K$\room & $H^{1/2}(\Gamma) \to H^{1/2}(\Gamma)$ & $\frac{|s|^{3/2}}{\sigma\underline\sigma^{3/2}}$ &
 $t^{3/2}\max\{1,t^{3/2}\} D_3(\varphi,t)$ & 2 & $\mathcal O(t)$\\
$\mathrm W$\room & $H^{1/2}(\Gamma) \to H^{-1/2}(\Gamma)$ & $\frac{|s|^2}{\sigma\underline\sigma}$ &
 $t^2\max\{1,t\} D_4(\varphi,t)$ & 4 & $\mathcal O(1)$\\
$\mathrm{NtD}$\room & $H^{-1/2}(\Gamma)\to H^{1/2}(\Gamma)$ &  $\frac{|s|}{\sigma\underline\sigma^2}$ & $t^2\max\{1,t^2\} D_3(\lambda,t)$ & 2 & $\mathcal O(t)$\\
$\mathrm{DtN}$\room & $H^{1/2}(\Gamma) \to H^{-1/2}(\Gamma)$ & $\frac{|s|^2}{\sigma\underline\sigma}$ &
 $t^2\max\{1,t\} D_4(\varphi,t)$ & 4 & $\mathcal O(1)$\\

\end{tabular}
\caption{Bounds obtained using Theorem \ref{th:7.1} and known estimates in the Laplace domain. The first line acts as a prototype and has to be read as
$\|\mathrm F(s)\|_{\mathcal L(X,Y)} \le C \times C_{\mathrm F}(\sigma) |s|^\mu$ and 
$\|(f*\eta)(t)\|_Y \le C \times E(t) \times \int_0^t \|\eta^{(k)}(\tau)\|_X\mathrm d\tau.$ Here $\sigma:=\mathrm{Re}s$ and $\underline\sigma:=\min\{1,\sigma\}$.The last two columns contain information given by Theorems \ref{th:3.1},  \ref{th:3.2}, \ref{th:6.1} and \ref{th:6.2}, indicating the highest order $n$ of differentiation of $\eta$ involved in the bounds for $f*\eta$ and the growth of the bound as a function of $t$ for large $t$.}
\end{center}\label{table:1}
\end{table}

\section{Basic results on some evolution equations}\label{sec:8}

\subsection{Homogeneous Dirichlet conditions}

In this section we gather some results concerning solutions of the non-homogeneous wave equation with homogeneous initial conditions and homogeneous Dirichlet boundary conditions on the ball $B_T$ introduced in Section \ref{sec:2}. We recall that $C_T$ is the Poincar\'e-Friedrichs constant in $B_T$ (see \eqref{eq:3.a}). The problem under consideration is:
\begin{subequations}\label{eq:8.3}
\begin{alignat}{4}
 \ddot u(t)&=\Delta u(t) + f(t) & \qquad & t\ge  0,\\
 \gamma_T u(t)&=0 & & t\ge 0,\\
 u(0)=\dot u(0)&=0.
\end{alignat}
\end{subequations}
We will deal with two different types of solutions of this problem. A strong solution is a function such that
\begin{equation}\label{eq:8.10}
u \in \mathcal C^2([0,\infty);L^2(B_T))\cap \mathcal C^1([0,\infty);H^1_0(B_T))\cap \mathcal C([0,\infty);H^2(B_T)),
\end{equation}
with the wave equation satisfied in $L^2(B_T)$ for all $t$. We will refer to a weak solution as a function such that
\begin{equation}\label{eq:8.11}
u \in \mathcal C^2([0,\infty);H^{-1}(B_T))\cap \mathcal C^1([0,\infty);L^2(B_T))\cap \mathcal C([0,\infty);H^1_0(B_T))
\end{equation}
and such that the wave equation is satisfied in $H^{-1}(B_T)$ (the dual space of $H^1_0(B_T)$) for all $t$. Note that the concept of weak solution relaxes both time and space regularity requirements and does not exactly coincide with the concept of mild solution given in \cite{Paz83} for example.

\begin{proposition}\label{prop:8.1}
Let $f:[0,\infty)\to H^1_0(B_T)$ be continuous. Then, problem \eqref{eq:8.3} has a unique strong solution
satisfying the bounds for all $t\ge 0$
\begin{eqnarray}
\label{eq:8.4}
\| u(t)\|_{B_T}\le C_T \int_0^ t\| f(\tau)\|_{B_T}\mathrm d\tau,\\
\label{eq:8.5}
\| \nabla u(t)\|_{B_T}\le  \int_0^ t\| f(\tau)\|_{B_T}\mathrm d\tau,\\
\label{eq:8.6}
\|\Delta  u(t)\|_{B_T}\le  \int_0^ t\|\nabla f(\tau)\|_{B_T}\mathrm d\tau.
\end{eqnarray}
\end{proposition}

\begin{proposition}\label{prop:8.2}
Let $f:[0,\infty)\to L^2(B_T)$ be continuous. Then problem \eqref{eq:8.3} has a unique weak solution, and the bound \eqref{eq:8.5} is still valid. Finally, the function $w(t):=\int_0^t u(\tau)\mathrm d\tau$ is continuous from $[0,\infty)$ to $H^2(B_T)$.
\end{proposition}

\begin{remark}\label{remark:8.1}
Propositions \ref{prop:8.1} and \ref{prop:8.2} still hold for the Dirichlet problem in the domain $B_T^+:=B_T\cap \Omega_+$ with the following modifications: the space $H^2(B_T)$ has to be substituted by $H^1_\Delta(B_T^+)$, and the constant $C_T$ in \eqref{eq:8.4} has to be substituted by the constant $E_T$ of the Poincar\'e-Friedrichs inequality \eqref{eq:6.2}.
\end{remark}

\subsection{Mixed conditions}

Let us now consider the set $B_T^+:=B_T \cap \Omega_+$ and the evolution problem
\begin{subequations}\label{eq:8.15}
\begin{alignat}{4}
 \ddot u(t)&=\Delta u(t) + f(t) & \qquad & t\ge  0,\\
 \gamma_T u(t)&=0 & & t\ge 0,\\
\partial_\nu^+ u(t)&=0 & & t\ge 0,\\
 u(0)=\dot u(0)&=0.
\end{alignat}
\end{subequations}
We thus consider the spaces $V_T$ given in \eqref{eq:6.1} and
\begin{eqnarray}\label{eq:8.14}
D_T &:=& \{ u \in V_T \,:\, \Delta u \in L^2(B_T^+), \quad \partial^+_\nu u=0\}\\
&=& \{ u \in V_T \cap H^1_\Delta(B_T^+)\,:\, (\nabla u,\nabla v)_{\Omega_+}+(u,v)_{\Omega_+}=0 \quad \forall v \in V_T\}.
\nonumber
\end{eqnarray}

\begin{proposition}\label{prop:8.3}
For $f\in \mathcal C([0,\infty);V_T)$, the initial value problem \eqref{eq:8.15} has a unique solution
\[
u \in \mathcal C^2([0,\infty);L^2(\Omega))\cap \mathcal C^1 ([0,\infty);V_T)\cap \mathcal C([0,\infty);D_T),
\]
satisfying
\begin{eqnarray}
\label{eq:8.20}
\| u(t)\|_{B_T^+}\le C_T \int_0^ t\| f(\tau)\|_{B_T^+}\mathrm d\tau,\\
\label{eq:8.21}
\| \nabla u(t)\|_{B_T^+}\le  \int_0^ t\| f(\tau)\|_{B_T^+}\mathrm d\tau,\\
\label{eq:8.22}
\|\Delta  u(t)\|_{B_T^+}\le  \int_0^ t\|\nabla f(\tau)\|_{B_T^+}\mathrm d\tau.
\end{eqnarray}
If $f\in \mathcal C([0,\infty);L^2(\Omega))$ there exists a unique weak solution of \eqref{eq:8.15} (that is, with the equation satisfied in $V'_T$)
\[
u \in \mathcal C^2([0,\infty);V'_T)\cap \mathcal C^1 ([0,\infty);L^2(\Omega))\cap \mathcal C([0,\infty);V_T),
\]
satisfying  \eqref{eq:8.21} and such that $w(t):=\int_0^ t u(\tau)\mathrm d\tau$ is in $\mathcal C([0,\infty);D_T)$.
\end{proposition}

\appendix

\section{Wave equations by separation of variables}\label{sec:A}

In this section we are going to give a direct proof of a generalization Propositions \ref{prop:8.1} and \ref{prop:8.2}. This proof will be based on direct arguments with generalized Fourier series and will allows us to obtain the needed uniform estimates of non-homogeneous evolution equation of the second order in terms of $L^1$ norms of the data. The Hilbert structure of the functional spaces is going to be used in depth, allowing us to obtain strong results that cannot be easily derived with a direct application of the best known results on the theory of semigroups of operators. This is not to say that these results do no exist, but we think it can be of interest (especially within the boundary integral community) to see a direct proof of these theorems based on functional analysis tools that are common for researchers  integral equations.

\subsection{Three lemmas about series}

In all the following results $X$ is a separable Hilbert space and $I:=[a,b]$ is a compact interval.

\begin{lemma}\label{lemma:A.1}
Assume that $c_n:I \to X$ are continuous, 
\begin{equation}\label{eq:A.1}
(c_n(t),c_m(t))_X=0 \qquad \forall n\neq m, \qquad \forall t \in I,
\end{equation}
and
\[
\| c_n(t)\|_X^2 \le M_n \qquad \forall t \in I, \quad \forall n \qquad \mbox{with}\qquad
\sum_{n=1}^\infty M_n < \infty.
\]
Then the series
\begin{equation}\label{eq:A.0}
c(t):=\sum_{n=1}^\infty c_n(t)
\end{equation}
converges uniformly in $t$ to a continuous function.
\end{lemma}

\begin{proof}
Let $s_N:=\sum_{n=1}^N c_n\in \mathcal C(I;X)$. For all $M>N$,
\[
\| s_M(t)-s_N(t)\|_X^2 = \sum_{n=N+1}^M \| c_n(t)\|_X^2 \le \sum_{n=N+1}^M M_n,
\]
which proves that $s_N(t)$ converges uniformly. Continuity of the limit is a direct consequence of the uniform convergence of the series.
\end{proof}

\begin{lemma}\label{lemma:A.2}
Assume that $c_n:I \to X$ are continuously differentiable, 
\begin{equation}\label{eq:A.2}
(c_n(t),c_m(\tau))_X=0 \qquad\forall n\neq m, \qquad \forall t,\tau \in I,
\end{equation}
and 
\[
\| c_n(t)\|_X^2+\|\dot c_n(t)\|_X^2 \le M_n \qquad \forall t \in I, \quad \forall n
\qquad \mbox{with}\qquad
\sum_{n=1}^\infty M_n < \infty.
\]
Then the uniformly convergent series \eqref{eq:A.0}
defines a $\mathcal C^1(I;X)$ function and it can be differentiated term by term.
\end{lemma}

\begin{proof}
The hypothesis \eqref{eq:A.2} implies \eqref{eq:A.1} as well as
\[
(\dot c_n(t),\dot c_m(t))_X=0 \qquad \forall n\neq m, \qquad \forall t \in I.
\]
If  $s_N:=\sum_{n=1}^N c_n\in \mathcal C^1(I;X)$, then
\[
\| s_M(t)-s_N(t)\|_X^2 +\| \dot s_M(t)-\dot s_N(t)\|_X^2 = \sum_{n=N+1}^M \| c_n(t)\|_X^2 \le \sum_{n=N+1}^M M_n,
\]
and therefore $s_N$ is Cauchy in $\mathcal C^1(I;X)$ and thus convergent. The fact that the derivatives of the series converges to the series of the derivatives is part of what convergence in $\mathcal C^1(I;X)$ means.
\end{proof}

\begin{lemma}\label{lemma:A.3} Let $f:I \to X$ be a continuous function and let $\{\phi_n\}$ be a Hilbert basis of $X$. Then
\[
f(t)=\sum_{n=1}^\infty (f(t),\phi_n)_X \phi_n
\]
uniformly in $t \in I$.
\end{lemma}

\begin{proof} Note first that for fixed $t$, $f(t)\in X$ can be expanded in the Hilbert basis, so convergence of the series is easy to prove. Next, consider the square of the norms of the $N$-th partial sums
\[
a_N(t):=\Big\| \sum_{n=1}^N (f(t),\phi_n)_X\phi_n\Big\|_X^2 = \sum_{n=1}^N |(f(t),\phi_n)_X|^2,
\]
which are continuous functions of $t$. The pointwise limit is $\| f(t)\|_X^2$, which is also a continuous function of $t$. Since the sequence $a_N$ is increasing, by Dini's Theorem, convergence $a_N \to \| f(\punto)\|_X^2$ is uniform. Finally
\[
\Big\| f(t)-\sum_{n=1}^N (f(t),\phi_n)_X\phi_n\Big\|_X^2 =\sum_{n=N+1}^\infty |(f(t),\phi_n)_X|^2 =
\|f(t)\|_X^2-a_N(t),
\]
which proves the uniform convergence of the series.
\end{proof}

\subsection{The Dirichlet spectral series of the Laplace operator}

Let $\Omega$ be a Lipschitz domain and consider the sequence of Dirichlet eigenvalues and eigenfunctions of the Laplace operator:
\[
\phi_n \in H^1_0(\Omega) \qquad -\Delta \phi_n=\lambda_n \phi_n.
\]
The sequence is taken with non-decreasing values of $\lambda_n$ and assuming
$(\phi_n,\phi_m)_\Omega=\delta_{nm},$ for all $m,n$, 
i.e., $L^2(\Omega)$-orthonormality of eigenfunctions. Thus, $\{\phi_n\}$ is a Hilbert basis of $L^2(\Omega)$ and consequently, for all $u\in L^2(\Omega)$
\begin{equation}\label{eq:B.1}
\| u\|_\Omega^2 = \sum_{n=1}^\infty |(u,\phi_n)_\Omega|^2
\end{equation}
and 
\begin{equation}\label{eq:B.2}
u=\sum_{n=1}^\infty (u,\phi_n)_\Omega \phi_n, 
\end{equation}
with convergence in $L^2(\Omega)$. Using the orthogonality 
$(\nabla \phi_n,\nabla\phi_m)_\Omega=\delta_{nm} \lambda_n,$ 
we can prove that
\[
H^1_0(\Omega)=\Big\{ u\in L^2(\Omega)\,:\, \sum_{n=1}^\infty \lambda_n |(u,\phi_n)_\Omega|^2 < \infty\Big\}
\]
and
\begin{equation}\label{eq:B.3}
\| \nabla u\|_\Omega^2 =\sum_{n=1}^\infty \lambda_n |(u,\phi_n)_\Omega|^2  \qquad \forall u\in H^1_0(\Omega).
\end{equation}
This expression gives a direct estimate of the corresponding Poincar\'e-Friedrichs inequality as
\[
\| u\|_\Omega \le \frac1{\sqrt{\min\lambda_n}} \|\nabla u\|_\Omega=:C_\circ \|\nabla u\|_\Omega \qquad \forall u \in H^1_0(\Omega).
\]
Moreover, if $u\in H^1_0(\Omega)$, the series representation \eqref{eq:B.1} converges in $H^1_0(\Omega)$.

The associated Green operator is the operator $G:L^2(\Omega) \to D(\Delta)$ given by
\[
u:=G f \quad \mbox{solution of}\quad  u\in H^1_0(\Omega), \quad -\Delta u = f \quad \mbox{in $\Omega$}.
\]
Here $D(\Delta):=\{ u \in H^1_0(\Omega)\,:\, \Delta u \in L^2(\Omega)\}.$ Note that for the case of a smooth domain $D(\Delta)=H^1_0(\Omega)\cap H^2(\Omega)$, although this fact will not be used in the sequel. The space $D(\Delta)$ is endowed with the norm $\| \Delta \cdot\|_\Omega$. The series representation of $G$ is given by the expression
\[
G f= \sum_{n=1}^\infty \lambda_n^{-1} (f,\phi_n)_\Omega\phi_n
\]
(with convergence in $L^2(\Omega)$). 
Picard's Criterion can then be used to show that $G$ is surjective and
\[
D(\Delta)=\Big\{ u \in L^2(\Omega)\,:\, \sum_{n=1}^\infty \lambda_n^2 |(u,\phi_n)_\Omega|^2\Big\}.
\]
Two more series representations are then directly available, one for the Laplacian
\[
-\Delta u = \sum_{n=1}^\infty \lambda_n (u,\phi_n)_\Omega \phi_n \qquad \forall u \in D(\Delta),
\]
(with convergence in $L^2(\Omega)$) and another one for its norm
\begin{equation}\label{eq:B.4}
\|\Delta u\|_\Omega^2 =\sum_{n=1}^\infty \lambda_n^2|(u,\phi_n)_\Omega|^2 \qquad \forall u \in D(\Delta).
\end{equation}

\subsection{Strong solutions of the wave equation}

 We start the section with a reminder of one of the possible versions of Duhamel's principle that will be useful in the sequel. Its proof is straightforward.

\begin{lemma}\label{lemma:C.1} Let $g:[0,\infty)\to \mathbb R$ be a continuous function, $\omega>0$ and define
\[
\alpha(t):=\int_0^t \omega^{-1} \sin(\omega(t-\tau)) g(\tau)\mathrm d \tau.
\]
Then $\alpha \in \mathcal C^2([0,\infty)$, $ \alpha(0)=\dot\alpha(0)=0$,
\[
\dot \alpha(t)=\int_0^ t \cos(\omega(t-\tau)) g(\tau) \mathrm d\tau
\]
and $\ddot\alpha(t)+\omega^2\alpha(t)=g(t)$ for all $t\ge 0$.
\end{lemma}

For notational convenience, we will write $\xi_n:=\sqrt\lambda_n$.

\begin{proposition}\label{prop:C.1}
Let $f:[0,\infty) \to H^1_0(\Omega)$ be a continuous function and consider the sequence
\[
u_n(t):=\left(\int_0^t \xi_n^{-1}\sin\big(\xi_n(t-\tau)\big)(f(\tau),\phi_n)_\Omega\mathrm d\tau\right)\phi_n, \qquad n\ge 1.
\]
Then, the function
\begin{equation}\label{eq:C.1}
u(t):=\sum_{n=1}^\infty u_n(t)
\end{equation}
satisfies
\begin{equation}\label{eq:C.7}
u \in \mathcal C^2([0,\infty);L^2(\Omega)) \cap \mathcal C^1([0,\infty);H^1_0(\Omega)) \cap \mathcal C([0,\infty);D(\Delta)).
\end{equation}
Moreover, $u$ is the unique strong solution of the following evolution equation:
\begin{equation}\label{eq:C.8}
\ddot u(t)=\Delta u(t)+f(t)\quad \forall t \ge 0, \qquad u(0)=\dot u(0)=0.
\end{equation}
\end{proposition}

\begin{proof}
As a direct consequence of Lemma \ref{lemma:C.1}, it follows that $u_n \in \mathcal C^2([0,\infty);X)$, where $X$ is any of  $L^2(\Omega), H^1_0(\Omega) $ or $D(\Delta)$. Also, for all $t\ge 0$
\begin{eqnarray*}
\dot u_n(t) &=& \left( \int_0^ t \cos\big(\xi_n (t-\tau)\big) (f(\tau),\phi_n)_\Omega\mathrm d\tau\right)\phi_n,\\
\ddot u_n(t) &=& (f(t),\phi_n)_\Omega\phi_n-\lambda_n u_n(t)=
 (f(t),\phi_n)_\Omega\phi_n+\Delta u_n(t),
\end{eqnarray*}
and
\begin{equation}\label{eq:C.2}
(u_n(t),u_m(\tau))_\Omega=(\nabla u_n(t),\nabla u_m(\tau))_\Omega = 0 \qquad \forall n\neq m, \qquad \forall t,\tau\ge 0.
\end{equation}
By \eqref{eq:B.4}, it follows that for $t\in [0,T]$,
\begin{eqnarray*}
\|\Delta u_n(t)\|_\Omega^2 &=& \lambda_n^2 \left| \int_0^t \xi_n^{-1}\sin\big(\xi_n(t-\tau)\big)(f(\tau),\phi_n)_\Omega\mathrm d\tau\right|^2\\
&\le & \lambda_n t \int_0^ t |(f(\tau),\phi_n)_\Omega|^2\mathrm d\tau\le T \int_0^T \lambda_n  |(f(\tau),\phi_n)_\Omega|^2\mathrm d\tau=: M_n^{(1)}.
\end{eqnarray*}
By the Monotone Convergence Theorem and \eqref{eq:B.3}, we easily show that
\[
\sum_{n=1}^\infty M_n^{(1)} = T\int_0^T \Big(\sum_{n=1}^\infty \lambda_n |(f(\tau),\phi_n)_\Omega|^2\Big)
\mathrm d\tau=T \int_0^ T \|\nabla f(\tau)\|_\Omega^2\mathrm d\tau.
\]
Thanks to these bounds and \eqref{eq:C.2} (recall that $\Delta u_n(t)=-\lambda_n u_n(t)$),
 Lemma \ref{lemma:A.1} can be now applied in the space $X=D(\Delta)$ and interval $I=[0,T]$ for arbitrary $T>0$ and we thus prove that
$u \in \mathcal C([0,\infty);D(\Delta)) \subset \mathcal C([0,\infty);H^1_0(\Omega))
\subset \mathcal C([0,\infty);L^2(\Omega)).$
Note that the series \eqref{eq:C.1} converges for all $t$ and therefore, using the fact that $u_n(0)=0$, it follows that $u(0)=0$. Note also that in particular
\begin{equation}\label{eq:C.3}
\Delta u(t)=\sum_{n=1}^\infty \Delta u_n(t)
\end{equation}
uniformly in $t\in [0,T]$ for arbitrary $T$.

In a second step, we use \eqref{eq:B.3} to bound
\begin{eqnarray*}
\|\nabla u_n(t)\|_\Omega^2+\|\nabla \dot u_n(t)\|_\Omega^2 &=&
\lambda_n \left|  \int_0^t \xi_n^{-1}\sin\big(\xi_n(t-\tau)\big)(f(\tau),\phi_n)_\Omega\mathrm d\tau\right|^2\\
&&+
\lambda_n \left|  \int_0^ t \cos\big(\xi_n (t-\tau)\big) (f(\tau),\phi_n)_\Omega\mathrm d\tau\right|^2\\
&\le & T \int_0^T (1+\lambda_n) |(f(\tau),\phi_n)_\Omega|^2\mathrm d\tau=:M_n^{(2)}.
\end{eqnarray*}
By the Monotone Convergence Theorem and the series representations of the norms \eqref{eq:B.1} and \eqref{eq:B.3}, we obtain
\[
\sum_{n=1}^\infty M_n^{(2)}=T \int_0^T \Big(\|\nabla f(\tau)\|_\Omega^2+\|f(\tau)\|_\Omega^2\Big)\mathrm d\tau.
\]
Using \eqref{eq:C.2}, we can apply Lemma \ref{lemma:A.2} in the space $X=H^1_0(\Omega)$ and the intervals $I=[0,T]$ to prove that
$u\in \mathcal C^1([0,\infty);H^1_0(\Omega))\subset \mathcal C^1([0,\infty);L^2(\Omega)). $
From this, it follows that $\dot u(0)=0$.

In a third step, we notice that by \eqref{eq:C.3} and Lemma \ref{lemma:A.3}
\[
\sum_{n=1}^\infty \Big( \Delta u_n(t)+(f(t),\phi_n)_\Omega\phi_n\Big)=\Delta u(t)+f(t),
\]
with convergence in $L^2(\Omega)$ uniformly in $t\in [0,T]$ for any $T$.
Since $\ddot u_n(t)=\Delta u_n(t)+(f(t),\phi_n)_\Omega\phi_n$,
it follows that the series of the second derivatives $L^2(\Omega)$-converges, uniformly in $t$, to a continuous function. Since the series of the first derivatives is $t$-uniformly $L^2(\Omega)$-convergent (it is actually $H^1_0(\Omega)$-convergent, as we have seen before), it follows that $\ddot u(t)=\Delta u(t)+f(t)$ for all $t\ge 0$, 
and that $\ddot u\in \mathcal C([0,\infty);L^2(\Omega))$.

Finally, if $u$ satisfies \eqref{eq:C.7} and the homogeneous wave equation
\begin{equation}\label{eq:C.9}
\ddot u(t)=\Delta u(t)\quad \forall t \ge 0, \qquad u(0)=\dot u(0)=0,
\end{equation}
then, a simple well-known energy argument shows that $u\equiv 0$, which proves uniqueness of strong solution to \eqref{eq:C.8}.
\end{proof}

\begin{proposition}\label{prop:C.2} Let $u$ be the function of Proposition \ref{prop:C.1}. Then, for all $t\ge 0$,
\begin{equation}\label{eq:C.12}
\| \Delta u(t)\|_\Omega \le \int_0^ t \| \nabla f(\tau)\|_\Omega\mathrm d\tau\qquad\mbox{and} \qquad \| \nabla u(t)\|_\Omega
\le \int_0^t\| f(\tau)\|_\Omega\mathrm d\tau.
\end{equation}
\end{proposition}

\begin{proof} For arbitrary $t>0$ consider the functions $g_n(\punto;t):[0,t]\to D(\Delta)$ given by
\[
g_n(\tau;t):=\xi_n^{-1}\sin(\xi_n(t-\tau))(f(\tau),\phi_n)_\Omega\phi_n.
\]
These functions are mutually orthogonal in $D(\Delta)$ and $H^1_0(\Omega)$.
Note that $\psi_n:=\lambda_n^{-1/2}\phi_n$ is a complete orthonormal set in $H^1_0(\Omega)$ and that
\[
(\nabla v,\nabla \psi_n)_\Omega=\lambda_n(v,\psi_n)_\Omega \qquad \forall n, \quad \forall v\in H^1_0(\Omega).
\]
It is then easy to prove the bounds
\begin{eqnarray}\label{eq:C.4}
\|\Delta g_n(\tau;t)\|_\Omega^2 &\le & |(\nabla f(\tau),\nabla \psi_n)_\Omega|^2 \qquad \forall \tau\in [0,t], \quad \forall n,\\
\label{eq:C.5}
\|\nabla g_n(\tau;t)\|_\Omega^2 &\le & |(f(\tau),\phi_n)_\Omega|^2 \qquad \forall \tau\in [0,t], \quad \forall n.
\end{eqnarray}
Note that by Lemma \ref{lemma:A.3}, the series
\begin{equation}\label{eq:C.6}
\sum_{n=1}^\infty |(\nabla f(\tau),\nabla \psi_n)_\Omega|^2=\|\nabla f(\tau)\|_\Omega^2, \qquad \mbox{and}\qquad\sum_{n=1}^\infty |(f(\tau),\phi_n)_\Omega|^2=\|f(\tau)\|_\Omega^2
\end{equation}
converge uniformly in $\tau\in [0,t]$. Using \eqref{eq:C.4} and \eqref{eq:C.6}, it is clear that 
\begin{equation}\label{eq:C.11}
[0,t]\ni \tau\longmapsto g(\tau;t):=\sum_{n=1}^\infty g_n(\tau;t)
\end{equation}
is well defined as a $D(\Delta)$-convergent series. Since convergence of the series \eqref{eq:C.6} it also follows that the series \eqref{eq:C.11} is $\tau$-uniformly convergent in $D(\Delta)$ and therefore in $H^1_0(\Omega)$ as well. Uniform convergence then allows to interchange summation and integral signs in the following equalities
\[
u(t)=\sum_{n=1}^\infty u_n(t)=\sum_{n=1}^\infty\int_0^t g_n(\tau;t)\mathrm d\tau =\int_0^t\sum_{n=1}^\infty g_n(\tau;t)\mathrm d\tau=\int_0^t g(\tau;t)\mathrm d\tau.
\]
Applying now \eqref{eq:C.4}, \eqref{eq:C.6}, and Bochner's Theorem in the space $D(\Delta)$,  it follows that
\[
\|\Delta u(t)\|_\Omega \le \int_0^t \|\Delta g(\tau;t)\|_\Omega \mathrm d\tau\le \int_0^t \|\nabla f(\tau)\|_\Omega \mathrm d\tau.
\]
Similarly, \eqref{eq:C.5},  \eqref{eq:C.6}, and Bochner's Theorem in $H^1_0(\Omega)$, prove that
\[
\|\nabla u(t)\|_\Omega \le \int_0^t \|\nabla g(\tau;t)\|_\Omega \mathrm d\tau\le \int_0^t \| f(\tau)\|_\Omega \mathrm d\tau,
\]
which finishes the proof.
\end{proof}

\subsection{Weak solutions of the wave equation}

In this section we deal with solutions of the evolution problem \eqref{eq:C.8} when $f:[0,\infty)\to L^2(\Omega)$ is continuous. In this case, we will understand the wave equation as taking place in $H^{-1}(\Omega)$ for all $t\ge 0$. We first make some precisions about dual spaces and operators.

As customary in the literature, we let $H^{-1}(\Omega)$ be the representation of the dual space of $H^1_0(\Omega)$ that is obtained when $L^2(\Omega)$ is identified with its own dual space. If we denote by $(\punto,\punto)_\Omega$ the corresponding representation of the $H^{-1}(\Omega)\times H^1_0(\Omega)$ duality product as an extension of the $L^2(\Omega)$ inner product, then
\begin{equation}\label{eq:D.1}
\| v\|_{-1}:=\sup_{0\neq u \in H^1_0(\Omega)}\frac{(v,u)_\Omega}{\|\nabla u\|_\Omega}=\Big( \sum_{n=1}^\infty \lambda_n^{-1} |(v,\phi_n)_\Omega|^2\Big)^{1/2}.
\end{equation}
The Laplace operator admits a unique extension $\Delta: H^1_0(\Omega) \to H^{-1}(\Omega)$ given by the duality product
\[
-(\Delta u,v)_\Omega =(\nabla u,\nabla v)_\Omega \qquad \forall u,v \in H^1_0(\Omega)
\]
and admitting the series representation
\[
-\Delta u = \sum_{n=1}^\infty  \lambda_n (u,\phi_n)_\Omega \phi_n \qquad \forall u \in H^1_0(\Omega),
\]
with convergence in $H^{-1}(\Omega)$. Here $\Delta$ is just the distributional Laplace operator.

\begin{proposition}\label{prop:D.1} Let $f:[0,\infty)\to L^2(\Omega)$ be continuous. Then the initial value problem \eqref{eq:C.8} has a unique solution with regularity
\begin{equation}\label{eq:D.2}
u \in \mathcal C^2([0,\infty);H^{-1}(\Omega)) \cap \mathcal C^1([0,\infty);L^2(\Omega)) \cap \mathcal C([0,\infty);H^1_0(\Omega)).
\end{equation}
This solution satisfies the bound
\[
\|\nabla u(t)\|_\Omega \le \int_0^t \|f(\tau)\|_\Omega \mathrm d\tau \qquad \forall t \ge 0.
\]
Finally the function $w(t):=\int_0^ t u(\tau)\mathrm d\tau$ is continuous from $[0,\infty)$ to $D(\Delta)$.
\end{proposition}

\begin{proof}
Consider first the operator
\[
G^{1/2} f:=\sum_{n=1}^\infty \lambda_n^{-1/2} (f,\phi_n)_\Omega\phi_n. 
\]
Because of the series representation of the norms (see \eqref{eq:B.1}, \eqref{eq:B.3}, \eqref{eq:B.4} and \eqref{eq:D.1}) it is simple to see that $G^{1/2}$ defines an isometric isomorphism from $H^{-1}(\Omega)$ to $L^2(\Omega)$, from $L^2(\Omega)$ to $H^1_0(\Omega)$ and from $H^1_0(\Omega)$ to $D(\Delta)$. It is also clear that $\Delta G^{-1/2}=G^{-1/2}\Delta$ as a bounded operator from $D(\Delta)$ to $H^{-1}(\Omega)$.

As a simple consequence of the above, $G^{1/2}f\in \mathcal C([0,\infty);H^1_0(\Omega))$ and the problem
\[
\ddot v(t)=\Delta v(t)+G^{1/2}f(t)\quad \forall t \ge 0, \qquad v(0)=\dot v(0)=0
\]
has a unique strong solution by Proposition \ref{prop:C.1}. We next define $u:=G^{-1/2}v$. By the relations between the norms given by $G^{\pm1/2}$ and by the regularity of $v$ given by Proposition \ref{prop:C.1}, it follows that $u$ satisfies \eqref{eq:D.2}. It is also clear that $u(0)=\dot u(0)=0$. Additionally,
\[
\ddot u(t)=G^{-1/2} \ddot v(t)=G^{-1/2}(\Delta v(t)+G^{1/2}f(t))=\Delta G^{-1/2}v(t)+f(t)=\Delta u(t)+f(t) \qquad \forall t \ge 0,
\]
which makes $u$ a weak solution of \eqref{eq:C.8}.  Also, by Proposition \ref{prop:C.2},
\[
\| \nabla u(t)\|_\Omega =\|\Delta v(t)\|_\Omega \le \int_0^ t\|\nabla G^{1/2}f(\tau)\|_\Omega\mathrm d\tau=\int_0^t \| f(\tau)\|_\Omega\mathrm d\tau.
\]

To prove uniqueness of weak solution, we note that if $u$ satisfies \eqref{eq:D.2} and the initial value problem \eqref{eq:C.9} (with the equation satisfies in $H^{-1}(\Omega)$), then $G^{1/2}u$ is a strong solution of \eqref{eq:C.9} and it is therefore identically zero.

Finally, it $u$ is the weak solution of \eqref{eq:C.8} and $w=\int_0^ t u$, then $w\in \mathcal C^2([0,\infty);L^2(\Omega))$ and it satisfies
\[
\Delta w(t)=\ddot w(t)-\int_0^t f(\tau)\mathrm d\tau \qquad \forall t.
\]
(This is an equality as elements of $H^{-1}(\Omega)$ for all $t$.)
However, the right hand side of the latter expression is a continuous function with values  in $L^2(\Omega)$ and therefore $w\in \mathcal C([0,\infty);D(\Delta))$.
\end{proof}

\subsection{A simple generalization}

Consider now a closed subspace $V$ such that $H^1_0(\Omega)\subset V\subset H^1(\Omega)$ and that $V$ does not contain non-zero constant functions, so that there exists $C_\circ$ such that $\| u\|_\Omega\le C_\circ \|\nabla u\|$ for all $u \in V.$
We then consider the set
\[
D:=\{ u \in V\,:\, \Delta u \in L^2(\Omega), (\nabla u,\nabla v)_\Omega+(\Delta u,v)_\Omega=0\quad\forall v \in H^1(\Omega)\},
\]
endowed with  the norm $\|\Delta \cdot\|_\Omega$. We can thus obtain a complete orthonormal set of eigenfunctions
\[
\phi_n \in D, \qquad -\Delta \phi_n=\lambda_n\phi_n.
\]
The entire theory can be repeated for these more general boundary conditions, substituting $H^1_0(\Omega)$ by $V$, $D(\Delta)$ by $D$ and $H^{-1}(\Omega)$ by the representation of $V'$ that arises from identifying $L^2(\Omega)$ with its dual. In this case $\Delta : V \to V'$ is not the distributional Laplacian since elements of $V'$ cannot be understood as distributions unless $V=H^1_0(\Omega)$. In any case, the results of Propositions \ref{prop:C.1}, \ref{prop:C.2} and \ref{prop:D.1} can be easily adapted to this new situation, namely. Proposition \ref{prop:8.3} is just a particular case.

\bibliographystyle{abbrv}
\bibliography{RefsTDIE}

\end{document}